\documentclass[ukrainian,english]{article}
\usepackage[T2A,T1]{fontenc}
\usepackage[latin9]{luainputenc}
\usepackage{mathtools}
\usepackage{amsmath}
\usepackage{amsthm}
\usepackage{amssymb}
\usepackage{wasysym}

\makeatletter
\theoremstyle{plain}
\newtheorem{thm}{\protect\theoremname}[section]
\theoremstyle{definition}
\newtheorem{defn}[thm]{\protect\definitionname}
\theoremstyle{definition}
\newtheorem{example}[thm]{\protect\examplename}
\theoremstyle{plain}
\newtheorem{prop}[thm]{\protect\propositionname}
\theoremstyle{remark}
\newtheorem{rem}[thm]{\protect\remarkname}
\theoremstyle{plain}
\newtheorem{lem}[thm]{\protect\lemmaname}

\makeatother

\usepackage{babel}
\addto\captionsenglish{\renewcommand{\definitionname}{Definition}}
\addto\captionsenglish{\renewcommand{\examplename}{Example}}
\addto\captionsenglish{\renewcommand{\lemmaname}{Lemma}}
\addto\captionsenglish{\renewcommand{\propositionname}{Proposition}}
\addto\captionsenglish{\renewcommand{\remarkname}{Remark}}
\addto\captionsenglish{\renewcommand{\theoremname}{Theorem}}
\addto\captionsukrainian{\renewcommand{\definitionname}{\inputencoding{koi8-u}Визначення}}
\addto\captionsukrainian{\renewcommand{\examplename}{\inputencoding{koi8-u}Приклад}}
\addto\captionsukrainian{\renewcommand{\lemmaname}{\inputencoding{koi8-u}Лема}}
\addto\captionsukrainian{\renewcommand{\propositionname}{\inputencoding{koi8-u}Твердження}}
\addto\captionsukrainian{\renewcommand{\remarkname}{\inputencoding{koi8-u}Примітка}}
\addto\captionsukrainian{\renewcommand{\theoremname}{\inputencoding{koi8-u}Теорема}}
\providecommand{\definitionname}{Definition}
\providecommand{\examplename}{Example}
\providecommand{\lemmaname}{Lemma}
\providecommand{\propositionname}{Proposition}
\providecommand{\remarkname}{Remark}
\providecommand{\theoremname}{Theorem}

\begin{document}
$ $\thispagestyle{empty}

\[
\mbox{BAR-ILAN UNIVERSITY}
\]

$ $

$ $

\[
\mbox{\Large{\textbf{Enumeration of Tableaux of Unusual Shapes}}}
\]

$ $

\[
\mbox{Amir Shoan}
\]

$\mbox{}$

$ $

$ $

$ $

\[
\begin{array}{c}
\normalsize{\mbox{Submitted in partial fulfillment of the requirements for the Master's Degree }}\\
\mbox{\normalsize{\mbox{ in the Department of Mathematics, Bar-Ilan University}} }
\end{array}
\]

$ $

$ $

$ $

$ $

$ $

$ $

$ $

$ $

$ $

$ $

$ $

$ $

\[
\mbox{\normalsize\mbox{}Ramat-Gan, Israel \ensuremath{} \ensuremath{}\ensuremath{} \ensuremath{} \ensuremath{} \ensuremath{} \ensuremath{} \ensuremath{} \ensuremath{} \ensuremath{} \ensuremath{} \ensuremath{} \ensuremath{} \ensuremath{} \ensuremath{} \ensuremath{} \ensuremath{} \ensuremath{} \ensuremath{} \ensuremath{} \ensuremath{} \ensuremath{} \ensuremath{} \ensuremath{} \ensuremath{} \ensuremath{} \ensuremath{} \ensuremath{} \ensuremath{} \ensuremath{} \ensuremath{} \ensuremath{} \ensuremath{} \ensuremath{} \ensuremath{} \ensuremath{} \ensuremath{} \ensuremath{} \ensuremath{} \ensuremath{}\ensuremath{} \ensuremath{} \ensuremath{} \ensuremath{} \ensuremath{} \ensuremath{} \ensuremath{} \ensuremath{} \ensuremath{} \ensuremath{} \ensuremath{} \ensuremath{} \ensuremath{} \ensuremath{} \ensuremath{} \ensuremath{} \ensuremath{} \ensuremath{} \ensuremath{} \ensuremath{} \ensuremath{} \ensuremath{} \ensuremath{} \ensuremath{} \ensuremath{} \ensuremath{} \ensuremath{} \ensuremath{} \ensuremath{} \ensuremath{} \ensuremath{} \ensuremath{}\ensuremath{} \ensuremath{} \ensuremath{} \ensuremath{} 2020}
\]

\newpage

\[
\begin{array}{c}
\mbox{This work was carried out under the supervision of }\\
\mbox{Prof. Ron M. Adin and Prof. \ensuremath{\mbox{ Yuval Roichman }}}\\
\mbox{ Department of Mathematics, Bar Ilian University}
\end{array}
\]

\thispagestyle{empty}

\newpage

\thispagestyle{empty}

\tableofcontents{}

\thispagestyle{empty}

\break

\normalsize

\section{Abstract}

\setcounter{page}{1}

In this thesis we enumerate standard young tableaux (SYT) of certain
truncated skew shapes, which we call battery shapes. This is motivated
by a chess problem. In an \foreignlanguage{ukrainian}{enumerative}
chess problem, the set of moves in the solution is (usually) unique,
but the order is not. The task of counting the feasible permutations
may be accomplished by solving an equivalent problem in enumerative
combinatorics. Almost all such problems have been of a special type
known as ``series movers''. In this thesis we use generalized hypergeometric
functions to enumerate SYT of battery shapes, and thus solve a chess
problem posed by Buchanan \cite{Bu}. 
\begin{defn}
Let $\lambda=\left(\lambda_{1},\lambda_{2},\ldots,\lambda_{t}\right)$
be a partition of positive integer n. Let $\left[\lambda\right]$
be a Young diagram of shape $\lambda$. By adding a column of length
$a$ above the upper cell of the $k$-th column of $[\lambda]$ we
construct a battery shape of size $n+a$. Denote this battery shape
by $\left[\lambda,a,k\right].$
\begin{defn}
A standard Young tableau of shape $[\lambda,a,k]$ and size $n+a$
is a bijection $T:[\lambda,a,k]\rightarrow\left[n+a\right]$ between
the Young diagram of shape $[\lambda,a,k]$ and the set $[n+a]:=\{1,2,\ldots,n+a\}$,
such that row entries increase from left to right and column entries
increase from top to bottom.\label{def:A-standard-Young}
\end{defn}

\end{defn}

\begin{example}
A standard Young tableau of shape $\left[\left(4^{3}\right),3,2\right]$ 

\[
\left(\begin{array}{cccc}
 & 2\\
 & 4\\
 & 6\\
1 & 7 & 9 & 10\\
3 & 8 & 11 & 13\\
5 & 12 & 14 & 15
\end{array}\right).
\]
\end{example}

In this thesis we use hypergeometric functions to find formulas for
the number of $SYT$ of battery shape $\left[\left(m^{n}\right),a,k\right]$
where $k$ is fixed and the other parameters $a,m,n$ vary. For example
we prove:
\begin{thm}
The number of SYT of battery shape $\left[\left(m^{n}\right),a,2\right]$
is equal to $f^{\left(m^{n}\right)}\mbox{}_{3}F_{2}\left(a,m,-n;1,-mn;1\right)$,
where $f^{\left(m^{n}\right)}$ is the number of $SYT$ of shape $\left(m^{n}\right)$
and $_{3}F_{2}$ is a generalized hypergeometric function. 
\end{thm}

This result yields explicit multiplicative formulas for specific values
of the parameters.

Buchanan asked an enumerative chess problem which is equivalent to
the enumeration of $SYT$ of battery shape $\left[\left(11^{7}\right),1,6\right]$.
By using hypergeometric functions we prove that the answer is

$2^{5}\times3^{2}\times5^{2}\times11\times13\times17^{2}\times19^{3}\times23^{2}\times29\times31$$\times37^{2}\times41\times3361178017$

$\times2839893182041.$

\break

\section{Background: partitions, diagrams and tableaux}

\subsection{Partitions and diagrams}

A partition of a positive integer $n$ is a weakly decreasing sequence
of positive integers summing to $n$: $\lambda=\left(\lambda_{1},\ldots,\lambda_{t}\right)$,
where $\lambda_{1}\geq\ldots\geq\lambda_{t}>0$ and 

$\lambda_{1}+\ldots+\lambda_{t}=n.$ This is denoted $\lambda\vdash n$.
\label{syt}

The Young diagram $\left[\lambda\right]$ corresponding to a partition
$\lambda=\left(\lambda_{1},\ldots,\lambda_{t}\right)$ is a collection
of cells, arranged in left justified rows, where the length of row
$i$ (from the top) is $\lambda_{i}$ $\left(1\leq i\leq t\right)$.

For example , the diagram corresponding to $\lambda=\left(5,3,1\right)$
is 

\[
\left[\begin{array}{ccccc}
\bullet & \bullet & \bullet & \bullet & \bullet\\
\bullet & \bullet & \bullet\\
\bullet
\end{array}\right].
\]

\subsection{Regular shapes }
\begin{defn}
A standard Young tableau $\left(SYT\right)$ of shape $\lambda\vdash n$
is a bijection $T:\left[\lambda\right]\rightarrow\left[n\right]$
between the set of cells in the Young diagram of shape $\lambda$
and the set $\left[n\right]\coloneqq\left\{ 1,2,\ldots,n\right\} $,
such that the entries are increasing from left to right in each row,
and from top to bottom in each column.
\end{defn}

The set of standard Young tableaux of shape $\lambda$ is denoted
by $SYT\left(\lambda\right)$.

Denote $f^{\lambda}:=\left|SYT\left(\lambda\right)\right|.$ The following
classical formula (Hook Length Formula) is due to Frame, Robinson
and Thrall. 
\begin{prop}
\cite[Theorem 3.10.2]{Sa} If $\lambda\vdash n$, then

\[
f^{\lambda}=\frac{n!}{\prod_{c\in\left[\lambda\right]}h_{c}}.
\]
$ $
\end{prop}

The  product ranges over all the cells $c$ in the Young diagram of
$\lambda$, where $h_{c}$, called the hook length of the cell $c\in\left[\lambda\right]$,
is the number of cells to the right of $c$ in the same row, plus
the number of cells below $c$ in the same column, plus 1 (the cell
$c$ itself).
\begin{example}
The number of standard Young tableaux of shape $(3,2,1)$ is equal
to $\frac{6!}{5\cdot3\cdot1\cdot3\cdot1\cdot1}=16$. An example of
one of those tableaux is:

\[
\left(\begin{array}{ccc}
1 & 2 & 5\\
3 & 4\\
6
\end{array}\right).
\]
\end{example}

\subsection{Skew shapes }

A skew shape is a pair of partitions $(\lambda,\mu)$ such that the
diagram of $\lambda$ contains the diagram of $\mu$. The skew shape
is denoted $\lambda/\mu$. If $\lambda=(\lambda_{1},\lambda_{2},\ldots,\lambda_{k})$
and $\mu=(\mu_{1},\mu_{2},\ldots,\mu_{l})$ with $l\leq k$ then containment
means that $\mu_{i}\leq\lambda_{i}$ for all $i$.

The corresponding skew diagram is the set of cells that belong to
the diagram of $\lambda$ but not to that of $\mu.$
\begin{defn}
A standard Young tableau of skew shape $\lambda/\mu$ (of size $n$)
is a bijection $T:\left[\lambda/\mu\right]\rightarrow\left[n\right]$
between the cells of the skew Young diagram of $\lambda/\mu$ and
the set $[n]:=\{1,2,\ldots,n\}$, such that the entries in each row
and column are increasing.
\end{defn}

\begin{example}
A standard young tableau of shape $(4,3,2)/(2,1)$

\[
\left(\begin{array}{cccc}
 &  & 1 & 6\\
 & 2 & 3\\
4 & 5
\end{array}\right).
\]
\end{example}

\subsection{Truncated shapes }
\begin{defn}
A diagram of truncated shape is a line-convex diagram obtained from
a Young diagram by deleting cells from the northeastern corner. 
\end{defn}

Truncated shapes were introduced and studied in \cite{AKR,Pa}. The
interest in the enumeration of standard Young tableaux of truncated
shapes is motivated by the enumeration of geodesics in flip graphs,
and of maximal chains in associated posets.
\begin{example}
A truncated ordinary shape $\lambda=\left[\left(5,5,2,1\right)\setminus\left(2\right)\right]$

\[
\left(\begin{array}{ccccc}
\bullet & \bullet & \bullet\\
\bullet & \bullet & \bullet & \bullet & \bullet\\
\bullet & \bullet\\
\bullet
\end{array}\right).
\]
\begin{example}
A truncated skew shape $\lambda=\left[\left(5,5,2,1\right)\setminus\left(2\right)/(2)\right]$

\[
\left(\begin{array}{ccccc}
 &  & \bullet\\
\bullet & \bullet & \bullet & \bullet & \bullet\\
\bullet & \bullet\\
\bullet
\end{array}\right).
\]
\end{example}

\end{example}

\subsection{Battery shapes }

This thesis is about the enumeration of SYT of battery shapes, which
are special truncated skew shapes; see Definitions 1.1 and 1.2 above.

\newpage

\section{Generalized hypergeometric functions }

In this thesis we prove that the number of $SYT$ of battery shape
$\left[\left(m^{n}\right),a,2\right]$ is equal to a generalized hypergeometric
function, and deduce formulas for other battery shapes.

We now introduce the relevant definitions and formulas.

Throughout this work $\mathbb{N}$ denotes the set of positive integers,
$\mathbb{Z}$ the set of integers, and $\mathbb{C}$ the field of
complex numbers. 
\begin{defn}
The Pochhammer symbol (or rising factorial) is defined by:

Let $a$ be a complex number and let $n$ be a non-negative integer.
Then:

\[
\left(a\right)_{n}=a\left(a+1\right)\left(a+2\right)\ldots\left(a+n-1\right)=\left(\begin{array}{c}
a+n-1\\
n
\end{array}\right)n!\mbox{ }\mbox{\ensuremath{\mbox{ }} }\mbox{ }\mbox{ }\mbox{ }\left(n\geq1\right)
\]

and

\[
\left(a\right)_{0}=1.
\]
\begin{defn}
Let $p$ and $q$ be positive integers, $a_{1},\ldots,a_{p},b_{1},\ldots,b_{q}\in\mathbb{C}$
such that $b_{1},\ldots,b_{q}$ are not non-positive integers. The
Corresponding generalized hypergeometric function is defined by

\[
_{p}F_{q}\left(a_{1},\ldots,a_{p};b_{1},\ldots,b_{q};z\right)=\sum_{n=0}^{\infty}\frac{\left(a_{1}\right)_{n}\cdots\left(a_{p}\right)_{n}}{\left(b_{1}\right)_{n}\cdots\left(b_{q}\right)_{n}}\frac{z^{n}}{n!}.
\]
\end{defn}

\end{defn}

\begin{rem}
The classical (non generalized) hypergeometric function is $_{2}F_{1}$.
\end{rem}

Usually a generalized hypergeometric function is an infinite power
series. In the following definition $a_{1},b_{1}$ are non-positive
integers satisfying $\left|a_{1}\right|\leq\left|b_{1}\right|$ and
the associated generalized hypergeometric function is a polynomial. 
\begin{defn}
Let $p$ and $q$ be positive integers, $a_{2},\ldots,a_{p},b_{2},\ldots,b_{q}\in\mathbb{N}$
and $a_{1},b_{1}$ non-positive integers satisfying $\left|a_{1}\right|\leq\left|b_{1}\right|.$
The corresponding generalized hypergeometric function is

\[
_{p}F_{q}\left(a_{1},\ldots,a_{p};b_{1},\ldots,b_{q};z\right)=\sum_{n=0}^{\left|a_{1}\right|}\frac{\left(a_{1}\right)_{n}\cdots\left(a_{p}\right)_{n}}{\left(b_{1}\right)_{n}\cdots\left(b_{q}\right)_{n}}\frac{z^{n}}{n!}.
\]
\end{defn}

\begin{rem}
There are two equivalent notations for generalized hypergeometric
functions, $_{p}F_{q}\left(a_{1},\ldots,a_{p};b_{1},\ldots,b_{q};z\right)$
and $_{p}F_{q}\left(\begin{array}{c}
a_{1},\ldots,a_{p}\\
b_{1},\ldots,b_{q}
\end{array};z\right)$.
\end{rem}

The term ``hypergeometric function'' was first used by Wallis. It
was studied by Euler and Gauss. We shall now state two classical theorems,
regarding $_{2}F_{1}$ and $_{3}F_{2}.$ For our application we need
these results for parameter values outside the classical range. We
shall therefore supply proofs for the versions of the theorems that
we actually use. 

The first theorem is due to Gauss
\begin{thm}
\cite[Theorem 15.4.20]{AsD} Let $a,b,c$ be complex numbers such
that $c,c-a,c-b,c-a-b\notin\left\{ 0,-1,-2,\ldots\right\} $ Then

\[
_{2}F_{1}\left(\begin{array}{c}
a\mbox{ }b\\
c
\end{array};1\right)=\frac{\Gamma\left(c\right)\Gamma\left(c-a-b\right)}{\varGamma\left(c-a\right)\Gamma\left(c-b\right)}.
\]
\end{thm}

We actually need the analogous formula when $c$ and $a$ are non-positive
integers.
\begin{thm}
Let $_{2}F_{1}\left(-a,b;-c;1\right)$ be a generalized hypergeometric
function \label{theorem 3.6} with parameters $a,b,c\in\mathbb{Z}$
satisfying $b>0$ , $0\leq a\leq c$ . Then:

\[
_{2}F_{1}\left(\begin{array}{c}
-a\mbox{ }b\\
-c
\end{array};1\right)=\frac{\left(\begin{array}{c}
c+b\\
a
\end{array}\right)}{\left(\begin{array}{c}
c\\
a
\end{array}\right)}.
\]
\label{thm:-10} 
\end{thm}

\begin{proof}
By induction on $a$. 

If $a=0$ we get indeed

\[
A_{0}=\mbox{}_{2}F_{1}\left(\begin{array}{cc}
-0 & b\\
-c
\end{array};1\right)=1=\frac{\left(\begin{array}{c}
c+b\\
0
\end{array}\right)}{\left(\begin{array}{c}
c\\
0
\end{array}\right)}.
\]

We assume that the claim holds for $a$ (and all admissible values
of $b$ and $c$) and prove it for $a+1$.

Assume:

\[
A_{a}=\mbox{}_{2}F_{1}\left(\begin{array}{cc}
-a & b\\
-c
\end{array};1\right)=\frac{\left(\begin{array}{c}
c+b\\
a
\end{array}\right)}{\left(\begin{array}{c}
c\\
a
\end{array}\right)}.
\]

By definition,

\[
A_{a}=\mbox{}_{2}F_{1}\left(\begin{array}{cc}
-a & b\\
-c
\end{array};1\right)=\sum_{n=0}^{a}\frac{\left(b\right)_{n}\left(-a\right)_{n}}{\left(-c\right)_{n}}\frac{1^{n}}{n!}.
\]

Recalling that for $n\in\mathbb{N}$ and any $z\in\mathbb{C}$:

\[
\left(\begin{array}{c}
-z+n-1\\
n
\end{array}\right)=\left(\begin{array}{c}
z\\
n
\end{array}\right)\left(-1\right)^{n}
\]

we get

\[
A_{a}=\sum_{n=0}^{a}\frac{\left(\begin{array}{c}
n+b-1\\
n
\end{array}\right)\left(\begin{array}{c}
a\\
n
\end{array}\right)}{\left(\begin{array}{c}
c\\
n
\end{array}\right)}.
\]

For similar reasons if $a+1\le c$,

\[
A_{a+1}=\mbox{}_{2}F_{1}\left(\begin{array}{cc}
-\left(a+1\right) & b\\
-c
\end{array};1\right)=\sum_{n=0}^{a+1}\frac{\left(\begin{array}{c}
n+b-1\\
n
\end{array}\right)\left(\begin{array}{c}
a+1\\
n
\end{array}\right)}{\left(\begin{array}{c}
c\\
n
\end{array}\right)}.
\]

Subtracting $A_{a}$ from $A_{a+1}$, and using Pascal's rule: 

\[
A_{a+1}-A_{a}=\sum_{n=1}^{a+1}\frac{\left(\begin{array}{c}
n+b-1\\
n
\end{array}\right)\left(\begin{array}{c}
a\\
n-1
\end{array}\right)}{\left(\begin{array}{c}
c\\
n
\end{array}\right)}
\]

\[
=\sum_{n=0}^{a}\frac{\left(\begin{array}{c}
n+b\\
n+1
\end{array}\right)\left(\begin{array}{c}
a\\
n
\end{array}\right)}{\left(\begin{array}{c}
c\\
n+1
\end{array}\right)}
\]

\[
=\sum_{n=0}^{a}\frac{b}{c}\frac{\left(\begin{array}{c}
n+b\\
n
\end{array}\right)\left(\begin{array}{c}
a\\
n
\end{array}\right)}{\left(\begin{array}{c}
c-1\\
n
\end{array}\right)}
\]

\[
=\frac{b}{c}\mbox{ }{}_{2}F_{1}\left(\begin{array}{cc}
-a & b+1\\
-c+1
\end{array};1\right).
\]

We proved, for $a+1\le c$ :

\[
_{2}F_{1}\left(\begin{array}{cc}
-\left(a+1\right) & b\\
-c
\end{array};1\right)=\mbox{}{}_{2}F_{1}\left(\begin{array}{cc}
-a & b\\
-c
\end{array};1\right)+\frac{b}{c}\mbox{ }{}_{2}F_{1}\left(\begin{array}{cc}
-a & b+1\\
-c+1
\end{array};1\right).
\]

By the induction hypothesis, we conclude that

\[
_{2}F_{1}\left(\begin{array}{cc}
-\left(a+1\right) & b\\
-c
\end{array};1\right)=\frac{\left(\begin{array}{c}
c+b\\
a
\end{array}\right)}{\left(\begin{array}{c}
c\\
a
\end{array}\right)}+\frac{b}{c}\frac{\left(\begin{array}{c}
c+b\\
a
\end{array}\right)}{\left(\begin{array}{c}
c-1\\
a
\end{array}\right)}
\]

\[
=\frac{\left(\begin{array}{c}
c+b\\
a
\end{array}\right)}{\left(\begin{array}{c}
c\\
a+1
\end{array}\right)}\left(\frac{c-a}{a+1}+\frac{b}{c}\left(\frac{c}{a+1}\right)\right)
\]

\[
=\frac{\left(\begin{array}{c}
c+b\\
a
\end{array}\right)}{\left(\begin{array}{c}
c\\
a+1
\end{array}\right)}\left(\frac{c+b-a}{a+1}\right)=\frac{\left(\begin{array}{c}
c+b\\
a+1
\end{array}\right)}{\left(\begin{array}{c}
c\\
a+1
\end{array}\right)}.
\]
\end{proof}

Here is another classical theorem.
\begin{thm}
\cite[Theorem 16.4.13]{OLBC} 

Let $a,b,c,d,e$ be complex numbers such that $Re\left(e-a\right)>0$ 

and $d+e-b-c\notin\left\{ 0,-1,-2,\ldots\right\} $. Then: 

\[
_{3}F_{2}\left(\begin{array}{ccc}
a, & b, & c\\
d, & e
\end{array};1\right)=
\]

\[
=\frac{c\left(e-a\right)}{de}\mbox{}_{3}F_{2}\left(\begin{array}{ccc}
a & b+1 & c+1\\
d+1 & e+1
\end{array};1\right)+\frac{d-c}{d}\mbox{}_{3}F_{2}\left(\begin{array}{ccc}
a & b+1 & c\\
d+1 & e
\end{array};1\right).
\]
\end{thm}

We actually need the analogous formula when $e-a$ is a non-positive
integer.
\begin{thm}
Let $_{3}F_{2}\left(\begin{array}{ccc}
a, & b, & -c\\
d, & -e
\end{array};1\right)$ be the generalized hypergeometric function with parameters $a,b,c,d,e\in\mathbb{Z}$
satisfying $a\geq0,b\geq-1,d>0$ and $0\leq c\leq e$. Then:\label{theorem 3.8}

\[
_{3}F_{2}\left(\begin{array}{ccc}
a, & b, & -c\\
d, & -e
\end{array};1\right)
\]

\[
=\frac{-c\left(e+a\right)}{de}\mbox{}_{3}F_{2}\left(\begin{array}{ccc}
a, & b+1 & -c+1\\
d+1 & -e+1
\end{array};1\right)+\frac{d+c}{d}\mbox{}_{3}F_{2}\left(\begin{array}{ccc}
a, & b+1 & -c\\
d+1, & -e
\end{array};1\right).
\]
\end{thm}

The proof will use two lemmas.

\break
\begin{lem}
For $a,b,c,d,e$ as above,

\[
_{3}F_{2}\left(a,b,-c;d,-e;1\right)=1+\frac{bc}{de}\sum_{t=1}^{a}\mbox{ \ensuremath{_{3}F_{2}\left(\begin{array}{c}
t,\mbox{ }b+1,\mbox{ \ensuremath{-c+1}}\mbox{ }\\
d+1,\mbox{ \ensuremath{-e+1}}
\end{array};1\right)}}
\]
\end{lem}

\begin{proof}
The binomial identity 

\[
\left(\begin{array}{c}
t+n-1\\
n
\end{array}\right)=\left(\begin{array}{c}
t+n-2\\
n
\end{array}\right)+\left(\begin{array}{c}
t+n-2\\
n-1
\end{array}\right)
\]

may be written, in terms of Pochhammer symbols, as

\[
\left(t\right)_{n}-\left(t-1\right)_{n}=n\left(t\right)_{n-1}
\]

Summation over values of $t$ from 1 to $a$ yields, for any positive
integers $a$ and $n:$

\[
\left(a\right)_{n}=n\sum_{t=1}^{a}\left(t\right)_{n-1}.
\]

Now, by definition 

\[
_{3}F_{2}\left(a,b,-c;d,-e;1\right)=\sum_{n=0}^{c}\frac{\left(a\right)_{n}\left(b\right)_{n}\left(-c\right)_{n}}{\left(d\right)_{n}\left(-e\right)_{n}n!}.
\]

Thus we get

\[
_{3}F_{2}\left(a,b,-c;d,-e;1\right)=1+\sum_{n=1}^{c}\frac{\left(a\right)_{n}\left(b\right)_{n}\left(-c\right)_{n}}{\left(d\right)_{n}\left(-e\right)_{n}n!}
\]

\[
=1+\sum_{n=1}^{c}\sum_{t=1}^{a}\frac{n\left(t\right)_{n-1}\left(b\right)_{n}\left(-c\right)_{n}}{\left(d\right)_{n}\left(-e\right)_{n}n!}
\]

\[
=1+\sum_{n=1}^{c}\sum_{t=1}^{a}\frac{\left(t\right)_{n-1}b\left(b+1\right)_{n-1}\left(-c\right)\left(-c+1\right)_{n-1}}{d\left(d+1\right)_{n-1}\left(-e\right)\left(-e+1\right)_{n}\left(n-1\right)!}
\]

\[
=1+\frac{bc}{de}\sum_{t=1}^{a}\mbox{}_{3}F_{2}\left(t,b+1,-c+1;d+1,-e+1;1\right).
\]
\end{proof}
\begin{lem}
The equality

\[
_{3}F_{2}\left(\begin{array}{ccc}
a, & b, & -c,\\
d, & -e,
\end{array};1\right)
\]

\[
=\frac{c\left(-e-a\right)}{de}{}_{3}F_{2}\left(\begin{array}{ccc}
a, & b+1, & -c+1\\
d+1, & -e+1,
\end{array};1\right)+\frac{d+c}{d}\mbox{ }_{3}F_{2}\left(\begin{array}{ccc}
a, & b+1, & -c\\
d+1, & -e
\end{array};1\right).
\]
\end{lem}

is equivalent to the equality 

\[
a+b\sum_{n=1}^{a}{}_{3}F_{2}\left(\begin{array}{ccc}
n, & b+1, & -c+1,\\
d+1, & -e+1,
\end{array};1\right)
\]

\[
=\sum_{n=1}^{a}\frac{\left(b+1\right)\left(d+c\right)}{\left(d+1\right)}{}_{3}F_{2}\left(\begin{array}{ccc}
n, & b+2, & -c+1,\\
d+2, & -e+1,
\end{array};1\right)
\]

\[
+\sum_{n=1}^{a}\frac{\left(b+1\right)\left(-e-a\right)\left(1-c\right)}{\left(d+1\right)\left(-e+1\right)}{}_{3}F_{2}\left(\begin{array}{ccc}
n, & b+2, & -c+2,\\
d+2, & -e+2,
\end{array};1\right).
\]

\begin{proof}
By Lemma 3.10,

\[
_{3}F_{2}\left(\begin{array}{ccc}
a, & b, & -c,\\
d, & -e,
\end{array};1\right)=1+\frac{bc}{de}\sum_{t=1}^{a}{}_{3}F_{2}\left(\begin{array}{c}
t,\mbox{ }b+1,\mbox{ \ensuremath{-c+1}}\mbox{ }\\
d+1,\mbox{ \ensuremath{-e+1}}
\end{array};1\right),
\]

\[
\frac{c\left(-e-a\right)}{de}{}_{3}F_{2}\left(\begin{array}{ccc}
a, & b+1, & -c+1\\
d+1, & -e+1,
\end{array};1\right)=
\]

\[
=\frac{c\left(-e-a\right)}{de}\left(1+\frac{\left(b+1\right)\left(-c+1\right)}{\left(d+1\right)\left(-e+1\right)}\sum_{t=1}^{a}{}_{3}F_{2}\left(\begin{array}{ccc}
t, & b+2, & -c+2\\
d+2, & -e+2,
\end{array};1\right)\right)
\]

and

\[
\frac{d+c}{d}\mbox{}_{3}F_{2}\left(\begin{array}{ccc}
a, & b+1, & -c\\
d+1, & -e
\end{array};1\right)=
\]

\[
=\frac{d+c}{d}\left(1+\frac{\left(b+1\right)\left(-c\right)}{\left(d+1\right)\left(-e\right)}\sum_{t=1}^{a}\mbox{ }_{3}F_{2}\left(\begin{array}{ccc}
t, & b+2, & -c+1\\
d+2, & -e+1
\end{array};1\right)\right).
\]

Since

\[
1-\frac{c\left(-e-a\right)}{de}-\frac{d+c}{d}=\frac{ac}{de}.
\]

we get that

\[
\frac{ac}{de}+\frac{bc}{de}\sum_{t=1}^{a}\mbox{ }{}_{3}F_{2}\left(\begin{array}{c}
t,\mbox{ }b+1,\mbox{ \ensuremath{-c+1}}\mbox{ }\\
d+1,\mbox{ \ensuremath{-e+1}}
\end{array};1\right)
\]

\[
=\frac{c\left(-e-a\right)\left(b+1\right)\left(-c+1\right)}{\left(d+1\right)\left(-e+1\right)de}\sum_{t=1}^{a}{}_{3}F_{2}\left(\begin{array}{ccc}
t, & b+2, & -c+2\\
d+2, & -e+2,
\end{array};1\right)
\]

\[
+\frac{\left(d+c\right)\left(b+1\right)\left(-c\right)}{d\left(d+1\right)\left(-e\right)}\sum_{t=1}^{a}\mbox{}_{3}F_{2}\left(\begin{array}{ccc}
t, & b+2, & -c+1\\
d+2, & -e+1
\end{array};1\right)
\]

Dividing by $\frac{c}{de}$ yields the required equality.
\end{proof}
$ $
\begin{proof}
$of$ $Theorem$ $3.9.$ By induction on $b$.

For $b=-1$, indeed

\[
1-\frac{ac}{de}=-\frac{c\left(a+e\right)}{de}+\frac{d+c}{d}.
\]

Assume that for a certain value of $b$ (and all admissible values
of the other parameters)

\[
_{3}F_{2}\left(\begin{array}{ccc}
a, & b, & -c,\\
d, & -e,
\end{array};1\right)
\]

\[
=\frac{c\left(-e-a\right)}{de}{}_{3}F_{2}\left(\begin{array}{ccc}
a, & b+1, & -c+1\\
d+1, & -e+1,
\end{array};1\right)+\frac{d+c}{d}\mbox{}_{3}F_{2}\left(\begin{array}{ccc}
a, & b+1, & -c\\
d+1, & -e
\end{array};1\right).
\]

By lemma 3.11, this is equivalent to

\[
a+b\sum_{n=1}^{a}{}_{3}F_{2}\left(\begin{array}{ccc}
n, & b+1, & -c+1,\\
d+1, & -e+1,
\end{array};1\right)
\]

\[
=\sum_{n=1}^{a}\frac{\left(b+1\right)\left(d+c\right)}{\left(d+1\right)}{}_{3}F_{2}\left(\begin{array}{ccc}
n, & b+2, & -c+1,\\
d+2, & -e+1,
\end{array};1\right)
\]

\[
-\sum_{n=1}^{a}\frac{\left(b+1\right)\left(-e-a\right)\left(c-1\right)}{\left(d+1\right)\left(-e+1\right)}{}_{3}F_{2}\left(\begin{array}{ccc}
n, & b+2, & -c+2,\\
d+2, & -e+2,
\end{array};1\right).
\]

By Lemma 3.10,

\[
_{3}F_{2}\left(\begin{array}{ccc}
t, & b+1, & -c+1,\\
d+1, & -e+1,
\end{array};1\right)
\]

\[
=1-\sum_{n=1}^{t}\frac{\left(b+1\right)\left(c-1\right)}{\left(d+1\right)\left(-e+1\right)}\mbox{ }_{3}F_{2}\left(\begin{array}{ccc}
n, & b+2, & -c+2,\\
d+2, & -e+2,
\end{array};1\right).
\]

Summing this for values of $t$ between $1$ and $a$ we get

\[
-a+\sum_{t=1}^{a}{}_{3}F_{2}\left(\begin{array}{ccc}
t, & b+1, & -c+1,\\
d+1, & -e+1,
\end{array};1\right)
\]

\[
=-\sum_{n=1}^{a}\frac{\left(a-n+1\right)\left(b+1\right)\left(c-1\right)}{\left(d+1\right)\left(-e+1\right)}{}_{3}F_{2}\left(\begin{array}{ccc}
n, & b+2, & -c+2,\\
d+2, & -e+2,
\end{array};1\right)
\]

Adding this to the previous formula we get

\[
\left(b+1\right)\sum_{n=1}^{a}{}_{3}F_{2}\left(\begin{array}{ccc}
n, & b+1, & -c+1,\\
d+1, & -e+1,
\end{array};1\right)
\]

\[
=\sum_{n=1}^{a}\frac{\left(b+1\right)\left(d+c\right)}{\left(d+1\right)}{}_{3}F_{2}\left(\begin{array}{ccc}
n, & b+2, & -c+1,\\
d+2, & -e+1,
\end{array};1\right)
\]

\[
-\sum_{n=1}^{a}\frac{\left(b+1\right)\left(-e-n+1\right)\left(c-1\right)}{\left(d+1\right)\left(-e+1\right)}{}_{3}F_{2}\left(\begin{array}{ccc}
n, & b+2, & -c+2,\\
d+2, & -e+2,
\end{array};1\right).
\]

Writing this equality for $a$ and for $a-1$ subtracting and dividing
by $b+1$, we get 

\[
_{3}F_{2}\left(\begin{array}{ccc}
a, & b+1, & -c+1,\\
d+1, & -e+1,
\end{array};1\right)
\]

\[
=\frac{\left(d+c\right)}{\left(d+1\right)}{}_{3}F_{2}\left(\begin{array}{ccc}
a, & b+2, & -c+1,\\
d+2, & -e+1,
\end{array};1\right)
\]

\[
-\frac{\left(-e-a+1\right)\left(c-1\right)}{\left(d+1\right)\left(-e+1\right)}{}_{3}F_{2}\left(\begin{array}{ccc}
a, & b+2, & -c+2,\\
d+2, & -e+2,
\end{array};1\right).
\]

This is the claimed equality for $b+1$ with $-c$, $-e$ and $d$
increased by 1
\end{proof}
\newpage

\section{Enumeration of $SYT$ of battery shape $\left[\left(m^{n}\right),a,2\right]$}

From now on we will give explicit formulas for the number of $SYT$
of battery shape $\left[\left(m^{n}\right),a,k\right]$ for small
values of $k.$ In this chapter we focus on the shape $\left[\left(m^{n}\right),a,2\right]$.
We will use the machinery of hypergeometric functions. 

\subsection{Main theorem \label{subsec:Enumerate-battery-tableau}}
\begin{thm}
The number of SYT of battery shape $\left[\left(m^{n}\right),a,2\right]$
is equal to 

\[
f^{\left(m^{n}\right)}\mbox{}_{3}F_{2}\left(a,m,-n;1,-mn;1\right)
\]

where $f^{\left(m^{n}\right)}$ is the number of $SYT$ of rectangular
shape $\left(m^{n}\right)$ and $_{3}F_{2}$ is a generalized hypergeometric
function.
\end{thm}

\begin{proof}
Consider the battery shape $\left[\left(m^{n}\right),a,2\right]$

\[
\begin{array}{c}
\left[\begin{array}{cccc}
 & \diamond\\
 & \diamond\\
 & \star\\
\bullet & \circ & \circ & \circ\\
\bullet & \circ & \circ & \circ\\
\circ & \circ & \circ & \circ
\end{array}\right]\\
\mbox{Figure 1}
\end{array}
\]

The cell immediately above the shape $\left(m^{n}\right)$ will be
called the pivot cell. It is denoted by a star $\left(\star\right)$
in Figure 1. We will sum over all possible entries of the pivot cell.
This entry is at least $a$.

If the entry is $a+t,$ then $t$ is the number of cells in the first
column of the shape whose entry is smaller than $a+t$. These entries
are denoted by bullets $\left(\bullet\right)$ in Figure 1. There
are also $a-1$ entries in the second column above the pivot cell,
whose entries are also smaller than $a+t$. All other entries in the
tableau are larger or equal to $a+t$. 

There are $\left(\begin{array}{c}
a+t-1\\
t
\end{array}\right)$ ways to choose which entries smaller than $a+t$ are in the first
column. In order to calculate the number of ways to place the entries
greater than the pivot cell we rotate the tableau $180^{o}$ degrees
and complement its entries by $i\leftrightarrow N+1-i$, where $N=m^{n}-t$
is the size of the rotated shape, whose entries are denoted by circles
$\left(\circ\right)$ in Figure 1. We get $SYT$ of shape $\left(m^{n-t},\left(m-1\right)^{t}\right)$.

In order to compute $f^{(m^{n-t},(m-1)^{t})}$, we will calculate
the ratio between this number and $f^{\left(m^{n}\right)}.$ 

Recall the Hook Length Formula from Proposition 2.2:

\[
f^{\lambda}=\frac{n!}{\prod_{c\in\left[\lambda\right]}h_{c}}.
\]

When we remove $t$ cells from the last column the following factors
are changed: the total number of cells, the hook lengths of the cells
in the last column and the hook lengths of the cells in the last $t$
rows. All other hook lengths do not change.

Thus, we obtain

\[
\frac{f^{\left(m^{n-t},\left(m-1\right)^{t}\right)}}{f^{\left(m^{n}\right)}}=
\]

\[
=\frac{\left(mn-t\right)!}{\left(n-t\right)!\prod_{i=1}^{m-1}\prod_{j=1}^{t}\left(i+j-1\right)}*\frac{n!\prod_{i=2}^{m}\prod_{j=1}^{t}\left(i+j-1\right)}{\left(mn\right)!}
\]

\[
=\frac{n!}{\left(n-t\right)!}\frac{\left(mn-t\right)!}{\left(mn\right)!}\frac{\prod_{j=1}^{t}\left(m+j-1\right)}{t!}
\]

\[
=\frac{\left(\begin{array}{c}
n\\
t
\end{array}\right)\left(\begin{array}{c}
m+t-1\\
t
\end{array}\right)}{\left(\begin{array}{c}
mn\\
t
\end{array}\right)}
\]

Taking the number of placements of letters smaller than $a+t$ into
account we deduce the following expression for the number of $SYT$
of battery shape $\left[\left(m^{n}\right),a,2\right]:$

\[
f^{\left(m^{n}\right)}\sum_{t=0}^{n}\frac{\left(\begin{array}{c}
t+m-1\\
t
\end{array}\right)\left(\begin{array}{c}
n\\
t
\end{array}\right)\left(\begin{array}{c}
t+a-1\\
t
\end{array}\right)}{\left(\begin{array}{c}
mn\\
t
\end{array}\right)}
\]

\[
=f^{\left(m^{n}\right)}\sum_{t=0}^{n}\frac{\left(\begin{array}{c}
t+m-1\\
t
\end{array}\right)\left(\begin{array}{c}
-n+t-1\\
t
\end{array}\right)\left(\begin{array}{c}
t+a-1\\
t
\end{array}\right)t!}{\left(\begin{array}{c}
-mn+t-1\\
t
\end{array}\right)\left(\begin{array}{c}
t\\
t
\end{array}\right)}\frac{1^{t}}{t!}.
\]

\[
=f^{\left(m^{n}\right)}\sum_{t=0}^{n}\frac{\left(m\right)_{t}\left(-n\right)_{t}\left(a\right)_{t}}{\left(-mn\right)_{t}\left(1\right)_{t}}\frac{1^{t}}{t!}
\]

\[
=f^{\left(m^{n}\right)}\mbox{ }_{3}F_{2}\left(a,m,-n;1,-mn;1\right).
\]
\end{proof}

\subsection{The computational algorithm }

The main result of the previous section expresses the number of $SYT$
of the required shape by a specific value of $_{3}F_{2}.$ In order
to get more explicit formulas we shall recursively use Theorem \ref{theorem 3.8},
increasing the value of the parameter $d$ until $d=a$. At this point
cancellation yields a $_{2}F_{1}$ function, for which Theorem 3.7
gives a closed formula, leading to a simplified expression\label{remark 13-1}.

For example,
\begin{enumerate}
\item One iteration:

\[
_{3}F_{2}\left(\begin{array}{ccc}
a, & b, & c\\
d, & e
\end{array};1\right)=
\]

\[
=\frac{c\left(e-a\right)}{de}\mbox{}_{3}F_{2}\left(\begin{array}{ccc}
a, & b+1 & c+1\\
d+1 & e+1
\end{array};1\right)+\frac{d-c}{d}\mbox{}_{3}F_{2}\left(\begin{array}{ccc}
a, & b+1 & c\\
d+1, & e
\end{array};1\right).
\]

\item Two iterations:

\[
_{3}F_{2}\left(\begin{array}{ccc}
a, & b, & c\\
d, & e
\end{array};1\right)=
\]

\[
\frac{c\left(e-a\right)\left(c+1\right)\left(e+1-a\right)}{\left(de\right)\left(d+1\right)\left(e+1\right)}\mbox{}_{3}F_{2}\left(\begin{array}{ccc}
a, & b+2 & c+2\\
d+2 & e+2
\end{array};1\right)
\]

\[
+\frac{2\left(d-c\right)c\left(e-a\right)}{\left(d+1\right)de}\mbox{}_{3}F_{2}\left(\begin{array}{ccc}
a, & b+2 & c+1\\
d+2, & e+1
\end{array};1\right)
\]

\[
+\frac{\left(d-c\right)\left(d+1-c\right)}{d\left(d+1\right)}\mbox{}_{3}F_{2}\left(\begin{array}{ccc}
a, & b+2 & c\\
d+2, & e
\end{array};1\right).
\]

\item Three iterations:
\end{enumerate}
\[
_{3}F_{2}\left(\begin{array}{ccc}
a, & b, & c\\
d, & e
\end{array};1\right)
\]

\[
=\frac{c\left(e-a\right)\left(c+1\right)\left(c+2\right)\left(e+1-a\right)\left(e+2-a\right)}{\left(de\right)\left(d+1\right)\left(e+1\right)\left(d+2\right)\left(e+2\right)}\mbox{}_{3}F_{2}\left(\begin{array}{ccc}
a, & b+3 & c+3\\
d+3 & e+3
\end{array};1\right)
\]

\[
+\frac{3c\left(e-a\right)\left(c+1\right)\left(e+1-a\right)\left(d-c\right)}{\left(de\right)\left(d+1\right)\left(e+1\right)\left(d+2\right)}\mbox{}_{3}F_{2}\left(\begin{array}{ccc}
a, & b+3 & c+2\\
d+3 & e+2
\end{array};1\right)
\]

\[
+\frac{3\left(d-c\right)c\left(e-a\right)\left(d-c+1\right)}{\left(d+1\right)de\left(d+2\right)}\mbox{}_{3}F_{2}\left(\begin{array}{ccc}
a, & b+3 & c+1\\
d+3 & e+1
\end{array};1\right)
\]

\[
+\frac{\left(d-c\right)\left(d-c+1\right)\left(d-c+2\right)}{d\left(d+1\right)\left(d+2\right)}\mbox{}_{3}F_{2}\left(\begin{array}{ccc}
a, & b+3 & c\\
d+3 & e
\end{array};1\right),
\]

and so on.

\subsection{Special case: fixed $a$}
\begin{rem}
$ $

\[
\frac{\left(\begin{array}{c}
c+b\\
b
\end{array}\right)}{\left(\begin{array}{c}
c+b-a\\
b
\end{array}\right)}=\frac{\left(\begin{array}{c}
c+b\\
a
\end{array}\right)}{\left(\begin{array}{c}
c\\
a
\end{array}\right)}
\]
\end{rem}

\textit{$Case$ }1. $\mbox{ }$ The number of $SYT$ of battery shape
$\left[\left(m^{n}\right),1,2\right]$ is

\[
f^{\left(m^{n}\right)}\frac{\left(\begin{array}{c}
mn+m\\
m
\end{array}\right)}{\left(\begin{array}{c}
mn+m-n\\
m
\end{array}\right)}.
\]

\begin{proof}
In order to prove this statement we apply Theorem 4.1:

\[
\mbox{}_{3}F_{2}\left(1,m,-n;1,-mn;1\right)={}_{2}F_{1}\left(m,-n;-mn;1\right)
\]

\[
=\frac{\left(\begin{array}{c}
mn+m\\
m
\end{array}\right)}{\left(\begin{array}{c}
mn+m-n\\
m
\end{array}\right)}.
\]

The last equality follows from Theorem 3.7 and Remark 4.2.
\end{proof}
\textit{$Case$ }$2$. $\mbox{ }$ The number of $SYT$ of battery
shape $\left[\left(m^{n}\right),2,2\right]$ is

\[
f^{\left(m^{n}\right)}\frac{\left(\begin{array}{c}
mn+m\\
m
\end{array}\right)}{\left(\begin{array}{c}
mn+m-n+1\\
m+1
\end{array}\right)}\frac{(2mn+m-n+1)}{\left(m+1\right)}.
\]

\begin{proof}
By Theorem 4.1 together with Theorems 3.9 and 3.7,

\[
\mbox{}_{3}F_{2}\left(2,m,-n;1,-mn;1\right)
\]

\[
=\frac{\left(-mn-2\right)}{m}\mbox{}_{3}F_{2}\left(\begin{array}{ccc}
2, & m+1 & -n+1\\
2 & -mn+1
\end{array};1\right)+\left(1+n\right)\mbox{}_{3}F_{2}\left(\begin{array}{ccc}
2, & m+1 & -n\\
2, & -mn
\end{array};1\right)
\]

\[
=\frac{\left(-mn-2\right)}{m}\mbox{}_{2}F_{1}\left(\begin{array}{cc}
m+1 & -n+1\\
-mn+1
\end{array};1\right)+\left(1+n\right)\mbox{}_{2}F_{1}\left(\begin{array}{cc}
m+1 & -n\\
-mn
\end{array};1\right)
\]

\[
=\frac{\left(-mn-2\right)}{m}\frac{\left(\begin{array}{c}
mn+m\\
m+1
\end{array}\right)}{\left(\begin{array}{c}
mn+m-n+1\\
m+1
\end{array}\right)}+\left(1+n\right)\frac{\left(\begin{array}{c}
mn+m+1\\
m+1
\end{array}\right)}{\left(\begin{array}{c}
mn+m-n+1\\
m+1
\end{array}\right)}
\]

\[
=\frac{\left(\begin{array}{c}
mn+m\\
m
\end{array}\right)}{\left(\begin{array}{c}
mn+m-n+1\\
m+1
\end{array}\right)}\left(\frac{n\left(-mn-2\right)}{m+1}+\frac{\left(n+1\right)\left(1+m+mn\right)}{m+1}\right)
\]

\[
=\frac{\left(\begin{array}{c}
mn+m\\
m
\end{array}\right)}{\left(\begin{array}{c}
mn+m-n+1\\
m+1
\end{array}\right)}\frac{(2mn+m-n+1)}{\left(m+1\right)}.
\]
\end{proof}
\textit{$Case$ }$3$. $\mbox{ }$ The number of $SYT$ of battery
shape $\left[\left(m^{n}\right),3,2\right]$ is

\[
f^{\left(m^{n}\right)}\frac{\left(\begin{array}{c}
mn+m\\
m
\end{array}\right)}{\left(\begin{array}{c}
mn-n+m+2\\
m+2
\end{array}\right)}\left(\frac{m^{2}(7n^{2}+7n+2)+m\left(-7n^{2}+9n+6\right)+2\left(n^{2}-3n+2\right)}{2(m+2)\left(m+1\right)}\right).
\]

\begin{proof}
Similar to the proof of case 2. Here two iterations of Theorem 3.9
are used.

\[
\mbox{}_{3}F_{2}\left(3,m,-n;1,-mn;1\right)
\]

\[
=\frac{-n\left(-mn-3\right)\left(-n+1\right)\left(-mn-2\right)}{2\left(-mn\right)\left(-mn+1\right)}\mbox{}_{3}F_{2}\left(\begin{array}{ccc}
3, & m+2 & -n+2\\
3 & -mn+2
\end{array};1\right)
\]

\[
+\frac{\left(n+1\right)-n\left(-mn-3\right)}{-mn}\mbox{}_{3}F_{2}\left(\begin{array}{ccc}
3, & m+2 & -n+1\\
3, & -mn+1
\end{array};1\right)
\]

\[
+\frac{\left(1+n\right)\left(2+n\right)}{2}\mbox{}_{3}F_{2}\left(\begin{array}{ccc}
3, & m+2 & -n\\
3, & -mn
\end{array};1\right)
\]

\[
=\frac{\left(-mn-3\right)\left(-n+1\right)\left(-mn-2\right)}{2m\left(-mn+1\right)}\mbox{}_{2}F_{1}\left(\begin{array}{cc}
m+2 & -n+2\\
-mn+2
\end{array};1\right)
\]

\[
+\frac{\left(n+1\right)\left(-mn-3\right)}{m}\mbox{}_{2}F_{1}\left(\begin{array}{cc}
m+2 & -n+1\\
-mn+1
\end{array};1\right)
\]

\[
+\frac{\left(1+n\right)\left(2+n\right)}{2}\mbox{}_{3}F_{2}\left(\begin{array}{cc}
m+2 & -n\\
-mn
\end{array};1\right)
\]

\[
=\frac{\left(-mn-3\right)\left(-n+1\right)\left(-mn-2\right)}{2m\left(-mn+1\right)}\frac{\left(\begin{array}{c}
mn+m\\
m+2
\end{array}\right)}{\left(\begin{array}{c}
mn+m-n+2\\
m+2
\end{array}\right)}
\]

\[
\frac{\left(n+1\right)\left(-mn-3\right)}{m}\frac{\left(\begin{array}{c}
mn+m+1\\
m+2
\end{array}\right)}{\left(\begin{array}{c}
mn+m-n+2\\
m+2
\end{array}\right)}
\]

\[
+\frac{\left(1+n\right)\left(2+n\right)}{2}\frac{\left(\begin{array}{c}
mn+m+2\\
m+2
\end{array}\right)}{\left(\begin{array}{c}
mn+m-n+2\\
m+2
\end{array}\right)}
\]

\[
=\frac{\left(\begin{array}{c}
mn+m\\
m
\end{array}\right)}{\left(\begin{array}{c}
mn-n+m+2\\
m+2
\end{array}\right)}\left(\frac{\left(n-1\right)n\left(mn+2\right)\left(mn+3\right)}{2\left(m+1\right)\left(m+2\right)}\right)
\]

\[
+\frac{\left(\begin{array}{c}
mn+m\\
m
\end{array}\right)}{\left(\begin{array}{c}
mn-n+m+2\\
m+2
\end{array}\right)}\left(\frac{\left(n+1\right)n\left(mn+m+1\right)\left(-mn-3\right)}{\left(m+1\right)\left(m+2\right)}\right)
\]

\[
+\frac{\left(\begin{array}{c}
mn+m\\
m
\end{array}\right)}{\left(\begin{array}{c}
mn-n+m+2\\
m+2
\end{array}\right)}\left(\frac{\left(n+1\right)\left(n+2\right)\left(mn+m+1\right)\left(mn+m+2\right)}{2\left(m+1\right)\left(m+2\right)}\right)
\]

\[
=\frac{\left(\begin{array}{c}
mn+m\\
m
\end{array}\right)}{\left(\begin{array}{c}
mn-n+m+2\\
m+2
\end{array}\right)}\left(\frac{m^{2}(7n^{2}+7n+2)+m\left(-7n^{2}+9n+6\right)+2\left(n^{2}-3n+2\right)}{2(m+2)\left(m+1\right)}\right).
\]
\end{proof}

\subsection{Special case: fixed $m$}

\textit{$Case$ }1. $\mbox{ }$ The number of $SYT$ of battery shape
$\left[\left(3^{n}\right),a,2\right]$ is

\[
f^{\left(3^{n}\right)}\frac{\left(\begin{array}{c}
3n+a\\
a
\end{array}\right)}{\left(\begin{array}{c}
2n+a+2\\
a+1
\end{array}\right)}\frac{\left(n+1\right)\left(an+2a+8n+4\right)}{2\left(2n+1\right)}.
\]

\begin{proof}
In order to prove this statement we use Theorem 4.1 and two iterations
of Theorem 3.9.

\[
\mbox{}_{3}F_{2}\left(3,a,-n;1,-3n;1\right)
\]

\[
=\frac{\left(-n-1\right)\left(-n+1\right)\left(-3n-2\right)}{2\left(-3n+1\right)}\mbox{}_{3}F_{2}\left(\begin{array}{ccc}
3, & a+2 & -n+2\\
3 & -3n+2
\end{array};1\right)
\]

\[
+\left(1+n\right)\left(-n-1\right)\mbox{}_{3}F_{2}\left(\begin{array}{ccc}
3, & a+2 & -n+1\\
3, & -3n+1
\end{array};1\right)
\]

\[
+\frac{\left(1+n\right)\left(2+n\right)}{2}\mbox{}_{3}F_{2}\left(\begin{array}{ccc}
3, & a+2 & -n\\
3, & -3n
\end{array};1\right)
\]

\[
=\frac{\left(-n-1\right)\left(-n+1\right)\left(-3n-2\right)}{2\left(-3n+1\right)}\mbox{}_{2}F_{1}\left(\begin{array}{cc}
a+2 & -n+2\\
-3n+2
\end{array};1\right)
\]

\[
+\left(1+n\right)\left(-n-1\right)\mbox{}_{2}F_{1}\left(\begin{array}{cc}
a+2 & -n+1\\
-3n+1
\end{array};1\right)
\]

\[
+\frac{\left(1+n\right)\left(2+n\right)}{2}\mbox{}_{2}F_{1}\left(\begin{array}{cc}
a+2 & -n\\
-3n
\end{array};1\right)
\]

\[
=\frac{\left(-n-1\right)\left(-n+1\right)\left(-3n-2\right)}{2\left(-3n+1\right)}\frac{\left(\begin{array}{c}
3n+a\\
a+2
\end{array}\right)}{\left(\begin{array}{c}
2n+a+2\\
a+2
\end{array}\right)}
\]

\[
+\left(1+n\right)\left(-n-1\right)\frac{\left(\begin{array}{c}
3n+a+1\\
a+2
\end{array}\right)}{\left(\begin{array}{c}
2n+a+2\\
a+2
\end{array}\right)}
\]

\[
+\frac{\left(1+n\right)\left(2+n\right)}{2}\frac{\left(\begin{array}{c}
3n+a+2\\
a+2
\end{array}\right)}{\left(\begin{array}{c}
2n+a+2\\
a+2
\end{array}\right)}
\]

\[
=\frac{\left(\begin{array}{c}
3n+a\\
a
\end{array}\right)}{\left(\begin{array}{c}
2n+a+2\\
a+1
\end{array}\right)}\left(\frac{3n\left(n-1\right)\left(n+1\right)\left(3n+2\right)}{2\left(a+1\right)\left(2n+1\right)}\right)
\]

\[
+\frac{\left(\begin{array}{c}
3n+a\\
a
\end{array}\right)}{\left(\begin{array}{c}
2n+a+2\\
a+1
\end{array}\right)}\left(\frac{3\left(-n-1\right)n\left(n+1\right)\left(a+3n+1\right)}{\left(a+1\right)\left(2n+1\right)}\right)
\]

\[
+\frac{\left(\begin{array}{c}
3n+a\\
a
\end{array}\right)}{\left(\begin{array}{c}
2n+a+2\\
a+1
\end{array}\right)}\left(\frac{\left(n+1\right)\left(n+2\right)\left(a+3n+1\right)\left(a+3n+2\right)}{2\left(a+1\right)\left(2n+1\right)}\right)
\]

\[
=\frac{\left(\begin{array}{c}
3n+a\\
a
\end{array}\right)}{\left(\begin{array}{c}
2n+a+2\\
a+1
\end{array}\right)}\frac{\left(n+1\right)\left(an+2a+8n+4\right)}{2\left(2n+1\right)}.
\]
\end{proof}

\paragraph{}

\textit{$Case$ }$2$. $\mbox{ }$ The number of $SYT$ of battery
shape $\left[\left(4^{n}\right),a,2\right]$ is

\[
f^{\left(4^{n}\right)}\frac{\left(\begin{array}{c}
4n+a\\
a
\end{array}\right)}{\left(\begin{array}{c}
3n+a+3\\
a+1
\end{array}\right)}\frac{\left(n+1\right)\left(a^{2}\left(n+2\right)\left(n+3\right)+a(n+2)\left(29n+15\right)+18\left(3n+1\right)\left(3n+2\right)\right)}{6\left(3n+1\right)\left(3n+2\right)}.
\]

\begin{proof}
In order to prove this statement we use three iterations of Theorem
3.9:

\[
\mbox{}_{3}F_{2}\left(4,a,-n;1,-4n;1\right)
\]

\[
=\frac{\left(-n-1\right)\left(-n+1\right)\left(-n+2\right)\left(-4n-3\right)\left(-4n-2\right)}{6\left(-4n+1\right)\left(-4n+2\right)}\mbox{}_{3}F_{2}\left(\begin{array}{ccc}
4, & a+3 & -n+3\\
4 & -4n+3
\end{array};1\right)+
\]

\[
+\frac{\left(-n-1\right)\left(-n+1\right)\left(-4n-3\right)\left(1+n\right)}{2\left(-4n+1\right)}\mbox{}_{3}F_{2}\left(\begin{array}{ccc}
4, & a+3 & -n+2\\
4 & -4n+2
\end{array};1\right)
\]

\[
+\frac{\left(1+n\right)\left(-n-1\right)\left(2+n\right)}{2}\mbox{}_{3}F_{2}\left(\begin{array}{ccc}
4, & a+3 & -n+1\\
4 & -4n+1
\end{array};1\right)
\]

\[
+\frac{\left(1+n\right)\left(2+n\right)\left(3+n\right)}{6}\mbox{}_{3}F_{2}\left(\begin{array}{ccc}
4, & a+3 & -n\\
4 & -4n
\end{array};1\right)
\]

\[
=\frac{\left(-n-1\right)\left(-n+1\right)\left(-n+2\right)\left(-4n-3\right)\left(-4n-2\right)}{6\left(-4n+1\right)\left(-4n+2\right)}\mbox{}_{2}F_{1}\left(\begin{array}{cc}
a+3 & -n+3\\
-4n+3
\end{array};1\right)
\]

\[
+\frac{\left(-n-1\right)\left(-n+1\right)\left(-4n-3\right)\left(1+n\right)}{2\left(-4n+1\right)}\mbox{}_{2}F_{1}\left(\begin{array}{cc}
a+3 & -n+2\\
-4n+2
\end{array};1\right)
\]

\[
+\frac{\left(1+n\right)\left(-n-1\right)\left(2+n\right)}{2}\mbox{}_{2}F_{1}\left(\begin{array}{cc}
a+3 & -n+1\\
-4n+1
\end{array};1\right)
\]

\[
+\frac{\left(1+n\right)\left(2+n\right)\left(3+n\right)}{6}\mbox{}_{2}F_{1}\left(\begin{array}{cc}
a+3 & -n\\
-4n
\end{array};1\right)
\]

\[
=\frac{\left(-n-1\right)\left(-n+1\right)\left(-n+2\right)\left(-4n-3\right)\left(-4n-2\right)}{6\left(-4n+1\right)\left(-4n+2\right)}\frac{\left(\begin{array}{c}
4n+a\\
a+3
\end{array}\right)}{\left(\begin{array}{c}
3n+a+3\\
a+3
\end{array}\right)}
\]

\[
+\frac{\left(-n-1\right)\left(-n+1\right)\left(-4n-3\right)\left(1+n\right)}{2\left(-4n+1\right)}\frac{\left(\begin{array}{c}
4n+a+1\\
a+3
\end{array}\right)}{\left(\begin{array}{c}
3n+a+3\\
a+3
\end{array}\right)}
\]

\[
+\frac{\left(1+n\right)\left(-n-1\right)\left(2+n\right)}{2}\frac{\left(\begin{array}{c}
4n+a+2\\
a+3
\end{array}\right)}{\left(\begin{array}{c}
3n+a+3\\
a+3
\end{array}\right)}
\]

\[
+\frac{\left(1+n\right)\left(2+n\right)\left(3+n\right)}{6}\frac{\left(\begin{array}{c}
4n+a+3\\
a+3
\end{array}\right)}{\left(\begin{array}{c}
3n+a+3\\
a+3
\end{array}\right)}
\]

\[
=\frac{\left(\begin{array}{c}
4n+a\\
a
\end{array}\right)}{\left(\begin{array}{c}
3n+a+3\\
a+1
\end{array}\right)}\left(\frac{-8n\left(n-1\right)\left(n-2\right)\left(n+1\right)\left(1+2n\right)\left(3+4n\right)}{6\left(a+1\right)\left(3n+1\right)\left(3n+2\right)}\right)
\]

\[
+\frac{\left(\begin{array}{c}
4n+a\\
a
\end{array}\right)}{\left(\begin{array}{c}
3n+a+3\\
a+1
\end{array}\right)}\left(\frac{12n\left(n+1\right)^{2}\left(1+4n+a\right)\left(4n+3\right)\left(n-1\right)}{6\left(a+1\right)\left(3n+1\right)\left(3n+2\right)}\right)
\]

\[
+\frac{\left(\begin{array}{c}
4n+a\\
a
\end{array}\right)}{\left(\begin{array}{c}
3n+a+3\\
a+1
\end{array}\right)}\left(\frac{-12n\left(n+1\right)^{2}\left(n+2\right)\left(1+4n+a\right)\left(a+4n+2\right)}{6\left(a+1\right)\left(3n+1\right)\left(3n+2\right)}\right)
\]

\[
+\frac{\left(\begin{array}{c}
4n+a\\
a
\end{array}\right)}{\left(\begin{array}{c}
3n+a+3\\
a+1
\end{array}\right)}\left(\frac{\left(1+n\right)\left(2+n\right)\left(3+n\right)\left(1+a+4n\right)\left(a+4n+2\right)\left(a+4n+3\right)}{6\left(a+1\right)\left(3n+1\right)\left(3n+2\right)}\right)
\]

\[
=\frac{\left(\begin{array}{c}
4n+a\\
a
\end{array}\right)}{\left(\begin{array}{c}
3n+a+3\\
a+1
\end{array}\right)}\frac{\left(n+1\right)\left(a^{2}\left(n+2\right)\left(n+3\right)+a(n+2)\left(29n+15\right)+18\left(3n+1\right)\left(3n+2\right)\right)}{6\left(3n+1\right)\left(3n+2\right)}.
\]
\end{proof}

\subsection{Special case: fixed $n$}

$Case\mbox{ }$$1.$ The number of $SYT$ of battery shape $\left[\left(m^{2}\right),a,2\right]$
is

\[
f^{\left(m^{2}\right)}\frac{\left(a+1\right)\left(a\left(m+1\right)+4\left(2m-1\right)\right)}{4\left(2m-1\right)}.
\]

$Case\mbox{ }$$2.$ The number of $SYT$ of battery shape $\left[\left(m^{3}\right),a,2\right]$
is

\[
f^{\left(m^{3}\right)}\frac{\left(a+1\right)\left(a^{2}\left(m+1\right)\left(m+2\right)+a\left(29m-14\right)\left(m+1\right)+18\left(3m-1\right)\left(3m-2\right)\right)}{18\left(3m-1\right)\left(3m-2\right)}.
\]

\newpage

\section{Additional shapes}

\subsection{Shape $\left[\left(m^{n}\right),a,3\right]$ }
\begin{defn}
We define a new function that will be useful in this section, a multiple
hypergeometric function: 

\[
_{\left[x_{0},x_{1},\ldots,x_{n}\right]}F_{\left[y_{0},y_{1}\ldots,y_{n}\right]}\left[\begin{array}{cccc}
a_{0,1},\ldots a_{0,x_{0}} & \brokenvert & b_{0,1},\ldots b_{0,y_{0}} & z_{0}\\
a_{1,1},\ldots a_{1,x_{1}} & \brokenvert & b_{1,1},\ldots b_{1,y_{1}}; & z_{1}\\
\vdots & \vdots & \vdots & \vdots\\
a_{n,1},\ldots a_{n,x_{n}} & \brokenvert & b_{n,1},\ldots b_{n,y_{n}} & z_{n}
\end{array}\right]=
\]

\[
=\sum_{m_{0}=0}^{\infty}\sum_{m_{1}=0}^{m_{0}}\cdots\sum_{m_{n=0}}^{m_{n-1}}\left(\prod_{i=0}^{n}\left(\frac{\prod_{j=1}^{x_{i}}\left(a_{i,j}\right)_{m_{i}}}{\prod_{j=1}^{y_{i}}\left(b_{i,j}\right)_{m_{i}}}\frac{z_{i}^{m_{i}}}{m_{i}!}\right)\right)
\]

Here $a_{i,j}$ and $b_{i,j}$ may depend on $m_{k}$ for $k<i.$
\end{defn}

\begin{rem}
When a suitable pair of parameters is non-positive, the summation
over $m_{0}$ is finite and the result is polynomial.
\begin{rem}
For $n=1$ we get the generalized hypergeometric function

\[
_{p}F_{q}\left[a_{1},a_{2}\ldots a_{p};b_{1},b_{2}\cdots b_{q};z\right]=\sum_{n=0}^{\infty}\frac{\prod_{j=1}^{p}\left(a_{i}\right)_{n}}{\prod_{j=1}^{q}\left(b_{i}\right)_{n}}\frac{z^{n}}{n!}
\]
\end{rem}

\end{rem}

\begin{example}
A battery shape $\left[\left(4^{3}\right),3,3\right]$ 

$\left[\begin{array}{cccc}
 &  & \diamond\\
 &  & \diamond\\
 &  & \star\\
\bullet & \bullet & \circ & \circ\\
\bullet & \circ & \circ & \circ\\
\circ & \circ & \circ & \circ
\end{array}\right]$.
\end{example}

\begin{thm}
The number of SYT of battery shape $\left[\left(m^{n}\right),a,3\right]$
is equal to 

$f^{\left(m^{n}\right)}\mbox{}_{\left[3,5\right]}F_{\left[2,4\right]}\left[\begin{array}{cccccccccc}
a & m & -n &  &  & \brokenvert & -mn & 1\\
a+t & m-1 & -n-1 & -t & -t & \brokenvert & -mn+t & -t-1 & -t-1 & 1
\end{array};\begin{array}{c}
1\\
1
\end{array}\right],$
\end{thm}

\paragraph{$Proof.$}
\begin{proof}
We use the same technique as in the proof of Theorem \ref{subsec:Enumerate-battery-tableau}.

Denote by a star $\left(\star\right)$ the pivot cell, namely the
cell immediately above the top cell in the third column of $\left[\left(m^{n}\right)\right]$.
Denote by bullets $\left(\bullet\right)$ the cells of $\left[\left(m^{n}\right)\right]$
whose entries are smaller than the entry of the pivot cell, and by
$\left(\circ\right)$ the cells of $\left[\left(m^{n}\right)\right]$
whose entries are larger than the entry of the pivot cell. The bullets
form a regular shape,with at most two columns. The circles form a
$180^{\circ}$ rotation of a regular shape. Let $t$ and $v$ be the
numbers of cells in the first and second column, respectively, of
the bullet tableau. By the Hook Length Formula we get:

The number of possible bullet tableaux is equal to $\frac{(t+v)!\left(t-v+1\right)}{v!(t+1)!}.$

number of possible circle tableaux is 

\[
f^{\left(m^{n}\right)}\frac{\left(\begin{array}{c}
t+m-1\\
t
\end{array}\right)\left(\begin{array}{c}
v+m-2\\
v
\end{array}\right)\left(\begin{array}{c}
n\\
t
\end{array}\right)t!\left(\begin{array}{c}
n+1\\
v
\end{array}\right)\left(\begin{array}{c}
t\\
v
\end{array}\right)v!}{\left(\begin{array}{c}
mn\\
t+v
\end{array}\right)\left(t+v\right)!\left(\begin{array}{c}
t+1\\
v
\end{array}\right)}.
\]

The product of these two numbers is

\[
f^{\left(m^{n}\right)}\frac{\left(\begin{array}{c}
t+m-1\\
t
\end{array}\right)\left(\begin{array}{c}
v+m-2\\
v
\end{array}\right)\left(\begin{array}{c}
n\\
t
\end{array}\right)\left(\begin{array}{c}
n+1\\
v
\end{array}\right)\left(\begin{array}{c}
t\\
v
\end{array}\right)\left(t-v+1\right)}{\left(\begin{array}{c}
mn\\
t+v
\end{array}\right)\left(\begin{array}{c}
t+1\\
v
\end{array}\right)\left(t+1\right)}.
\]

Multiplying by the number of ways to choose the set of entries of
the bullet cells and summing over all possible values of $t$ and
$v$ we get

\[
f^{\left(m^{n}\right)}\sum_{t=0}^{n}\sum_{v=0}^{t}\frac{\left(\begin{array}{c}
a+t+v-1\\
t+v
\end{array}\right)\left(\begin{array}{c}
t+m-1\\
t
\end{array}\right)\left(\begin{array}{c}
v+m-2\\
v
\end{array}\right)\left(\begin{array}{c}
n\\
t
\end{array}\right)\left(\begin{array}{c}
n+1\\
v
\end{array}\right)\left(\begin{array}{c}
t\\
v
\end{array}\right)\left(t-v+1\right)}{\left(\begin{array}{c}
mn\\
t+v
\end{array}\right)\left(\begin{array}{c}
t+1\\
v
\end{array}\right)\left(t+1\right)}
\]

$=f^{\left(m^{n}\right)}\sum_{t=0}^{n}\sum_{v=0}^{t}\frac{\left(\begin{array}{c}
-a\\
t+v
\end{array}\right)\left(\begin{array}{c}
-m\\
t
\end{array}\right)\left(\begin{array}{c}
-m+1\\
v
\end{array}\right)\left(\begin{array}{c}
n\\
t
\end{array}\right)\left(\begin{array}{c}
n+1\\
v
\end{array}\right)\left(\begin{array}{c}
t\\
v
\end{array}\right)^{2}}{\left(\begin{array}{c}
mn\\
t+v
\end{array}\right)\left(\begin{array}{c}
t+1\\
v
\end{array}\right)^{2}}$

$=f^{\left(m^{n}\right)}\sum_{t=0}^{n}\sum_{v=0}^{t}\frac{\left(\begin{array}{c}
-a\\
t
\end{array}\right)\left(\begin{array}{c}
-a-t\\
v
\end{array}\right)\left(\begin{array}{c}
-m\\
t
\end{array}\right)\left(\begin{array}{c}
-m+1\\
v
\end{array}\right)\left(\begin{array}{c}
n\\
t
\end{array}\right)\left(\begin{array}{c}
n+1\\
v
\end{array}\right)\left(\begin{array}{c}
t\\
v
\end{array}\right)^{2}}{\left(\begin{array}{c}
mn\\
t
\end{array}\right)\left(\begin{array}{c}
mn-t\\
v
\end{array}\right)\left(\begin{array}{c}
t+1\\
v
\end{array}\right)^{2}}$

$=f^{\left(m^{n}\right)}\sum_{t=0}^{n}\frac{\left(\begin{array}{c}
-a\\
t
\end{array}\right)\left(\begin{array}{c}
-m\\
t
\end{array}\right)\left(\begin{array}{c}
n\\
t
\end{array}\right)}{\left(\begin{array}{c}
mn\\
t
\end{array}\right)}\mbox{ \ensuremath{_{5}F_{4}\left(\begin{array}{ccccc}
a+t & m-1 & -n-1 & -t & -t\\
-mn+t & -t-1 & -t-1 & 1
\end{array};1\right)}}$

$=f^{\left(m^{n}\right)}\mbox{}_{\left[3,5\right]}F_{\left[2,4\right]}\left[\begin{array}{cccccccccc}
a & m & -n &  &  & \brokenvert & -mn & 1\\
a+t & m-1 & -n-1 & -t & -t & \brokenvert & -mn+t & -t-1 & -t-1 & 1
\end{array};\begin{array}{c}
1\\
1
\end{array}\right].$
\end{proof}
$ $

As an example, we give explicit formulas for small values of $n.$

$Case$ $1.$ The number of $SYT$ of battery shape $\left[\left(m^{2}\right),a,3\right]$
is

\[
f^{\left(m^{2}\right)}\frac{\left(a+1\right)\left(a+2\right)\left(a^{2}m\left(m+1\right)+a\left(m+1\right)\left(19m-24\right)+24\left(2m-1\right)\left(2m-3\right)\right)}{48\left(2m-1\right)\left(2m-3\right)}.
\]

$Case$ $2.$ The number of $SYT$ of battery shape $\left[\left(m^{3}\right),a,3\right]$
is

\[
f^{\left(m^{3}\right)}\frac{\left(a+1\right)\left(a+2\right)w}{1296\left(3m-1\right)\left(3m-2\right)\left(3m-4\right)\left(3m-5\right)}
\]

where 

\[
w=a^{4}m\left(m+1\right)^{2}\left(m+2\right)+6a^{3}m\left(11m-13\right)\left(m+1\right)\left(m+2\right)
\]

\[
+a^{2}\left(m+1\right)^{2}\left(1559m^{2}-3722m+2160\right)
\]

\[
+6a\left(m+1\right)\left(2521m^{3}-8169m^{2}+8078m-2280\right)
\]

\[
+648\left(3m-1\right)\left(3m-2\right)\left(3m-4\right)\left(3m-5\right).
\]

\subsection{Shape $\left[\left(m^{n}\right),a,4\right]$ }
\begin{thm}
The number of SYT of battery shape $\left[\left(m^{n}\right),a,4\right]$
is equal to 

$f^{\left(m^{n}\right)}\mbox{}_{\left[3,5,7\right]}F_{\left[2,4,6\right]}$

$\begin{array}{c}
\lceil\\
\brokenvert\\
\lfloor
\end{array}\begin{array}{ccccccc}
a & m & -n\\
a+t & m-1 & -n-1 & -t & -t\\
a+t+v & m-2 & -n-2 & -t-1 & -t-1 & -v & -v
\end{array}\begin{array}{c}
\brokenvert\\
\brokenvert\\
\brokenvert
\end{array}$

$\begin{array}{cccccc}
-mn & 1\\
-mn+t & -t-1 & -t-1 & 1\\
-mn+t+v & -t-2 & -t-2 & -v-1 & -v-1 & 1
\end{array};\begin{array}{c}
1\\
1\\
1
\end{array}\begin{array}{c}
\rceil\\
\brokenvert\\
\rfloor
\end{array}.$
\end{thm}

\begin{proof}
We use the Hook Length Formula.
\begin{enumerate}
\item The number of bullet tableaux is equal to $\frac{(t+v+w)!\left(v-w+1\right)\left(t-v+1\right)\left(t-w+2\right)}{w!\left(v+1\right)!\left(t+2\right)!}.$
\item The number of circle tableaux is equal to 

\[
f^{\left(m^{n}\right)}\frac{\left(\begin{array}{c}
t+m-1\\
t
\end{array}\right)\left(\begin{array}{c}
v+m-2\\
v
\end{array}\right)v!t!w!}{\left(\begin{array}{c}
mn\\
t+v+w
\end{array}\right)(t+v+w)!\left(\begin{array}{c}
t+1\\
v
\end{array}\right)\left(\begin{array}{c}
t+2\\
w
\end{array}\right)\left(\begin{array}{c}
v+1\\
w
\end{array}\right)}
\]

$\left(\left(\begin{array}{c}
w+m-3\\
w
\end{array}\right)\left(\begin{array}{c}
n\\
t
\end{array}\right)\left(\begin{array}{c}
n+1\\
v
\end{array}\right)\left(\begin{array}{c}
t\\
v
\end{array}\right)\left(\begin{array}{c}
v\\
w
\end{array}\right)\left(\begin{array}{c}
n+2\\
w
\end{array}\right)\left(\begin{array}{c}
t+1\\
w
\end{array}\right)\right).$

$ $

The total number is therefore

\[
f^{\left(m^{n}\right)}\sum_{t=0}^{n}\sum_{v=0}^{t}\sum_{w=0}^{v}\frac{\left(v-w+1\right)\left(t-v+1\right)\left(t-w+2\right)\left(\begin{array}{c}
t+m-1\\
t
\end{array}\right)\left(\begin{array}{c}
v+m-2\\
v
\end{array}\right)}{\left(\begin{array}{c}
mn\\
t+v+w
\end{array}\right)\left(\begin{array}{c}
t+1\\
v
\end{array}\right)\left(v+1\right)\left(t+1\right)\left(t+2\right)\left(\begin{array}{c}
t+2\\
w
\end{array}\right)\left(\begin{array}{c}
v+1\\
w
\end{array}\right)}
\]

$\left(\left(\begin{array}{c}
w+m-3\\
w
\end{array}\right)\left(\begin{array}{c}
n\\
t
\end{array}\right)\left(\begin{array}{c}
n+1\\
v
\end{array}\right)\left(\begin{array}{c}
t\\
v
\end{array}\right)\left(\begin{array}{c}
v\\
w
\end{array}\right)\left(\begin{array}{c}
a+t+v+w-1\\
t+v+w
\end{array}\right)\left(\begin{array}{c}
n+2\\
w
\end{array}\right)\left(\begin{array}{c}
t+1\\
w
\end{array}\right)\right)$

\[
=f^{\left(m^{n}\right)}\sum_{t=0}^{n}\sum_{v=0}^{t}\sum_{w=0}^{v}\frac{\left(\begin{array}{c}
-m\\
t
\end{array}\right)\left(\begin{array}{c}
-m+1\\
v
\end{array}\right)\left(\begin{array}{c}
-a-t-v\\
w
\end{array}\right)\left(\begin{array}{c}
-a-t\\
v
\end{array}\right)\left(\begin{array}{c}
-a\\
t
\end{array}\right)}{\left(\begin{array}{c}
mn-t-v\\
w
\end{array}\right)\left(\begin{array}{c}
mn-t\\
v
\end{array}\right)\left(\begin{array}{c}
mn\\
t
\end{array}\right)\left(\begin{array}{c}
t+1\\
v
\end{array}\right)^{2}\left(\begin{array}{c}
t+2\\
w
\end{array}\right)^{2}\left(\begin{array}{c}
v+1\\
w
\end{array}\right)^{2}}
\]

$\left(\left(\begin{array}{c}
-m+2\\
w
\end{array}\right)\left(\begin{array}{c}
n\\
t
\end{array}\right)\left(\begin{array}{c}
n+1\\
v
\end{array}\right)\left(\begin{array}{c}
t\\
v
\end{array}\right)^{2}\left(\begin{array}{c}
v\\
w
\end{array}\right)^{2}\left(\begin{array}{c}
n+2\\
w
\end{array}\right)\left(\begin{array}{c}
t+1\\
w
\end{array}\right)^{2}\right)$

\[
=f^{\left(m^{n}\right)}\sum_{t=0}^{n}\sum_{v=0}^{t}\frac{\left(\begin{array}{c}
-m\\
t
\end{array}\right)\left(\begin{array}{c}
-m+1\\
v
\end{array}\right)\left(\begin{array}{c}
-a-t\\
v
\end{array}\right)\left(\begin{array}{c}
-a\\
t
\end{array}\right)\left(\begin{array}{c}
n\\
t
\end{array}\right)\left(\begin{array}{c}
n+1\\
v
\end{array}\right)\left(\begin{array}{c}
t\\
v
\end{array}\right)^{2}}{\left(\begin{array}{c}
mn-t\\
v
\end{array}\right)\left(\begin{array}{c}
mn\\
t
\end{array}\right)\left(\begin{array}{c}
t+1\\
v
\end{array}\right)^{2}}
\]

\[
\sum_{w=0}^{v}\frac{\left(\begin{array}{c}
-a-t-v\\
w
\end{array}\right)\left(\begin{array}{c}
-m+2\\
w
\end{array}\right)\left(\begin{array}{c}
v\\
w
\end{array}\right)^{2}\left(\begin{array}{c}
n+2\\
w
\end{array}\right)\left(\begin{array}{c}
t+1\\
w
\end{array}\right)^{2}w!}{\left(\begin{array}{c}
mn-t-v\\
w
\end{array}\right)\left(\begin{array}{c}
t+2\\
w
\end{array}\right)^{2}\left(\begin{array}{c}
v+1\\
w
\end{array}\right)^{2}\left(\begin{array}{c}
w\\
w
\end{array}\right)}\frac{1}{w!}
\]

$=f^{\left(m^{n}\right)}\sum_{t=0}^{n}\sum_{v=0}^{t}\frac{\left(\begin{array}{c}
-m\\
t
\end{array}\right)\left(\begin{array}{c}
-m+1\\
v
\end{array}\right)\left(\begin{array}{c}
-a-t\\
v
\end{array}\right)\left(\begin{array}{c}
-a\\
t
\end{array}\right)\left(\begin{array}{c}
n\\
t
\end{array}\right)\left(\begin{array}{c}
n+1\\
v
\end{array}\right)\left(\begin{array}{c}
t\\
v
\end{array}\right)^{2}}{\left(\begin{array}{c}
mn-t\\
v
\end{array}\right)\left(\begin{array}{c}
mn\\
t
\end{array}\right)\left(\begin{array}{c}
t+1\\
v
\end{array}\right)^{2}}$

$_{7}F_{6}\left(\begin{array}{ccccccc}
a+t+v & m-2 & -v & -v & -n-2 & -t-1 & -t-1\\
-mn+t+v & -t-2 & -t-2 & -v-1 & -v-1 & 1
\end{array};1\right)$

$=f^{\left(m^{n}\right)}\mbox{}_{\left[3,5,7\right]}F_{\left[2,4,6\right]}$

$\begin{array}{c}
\lceil\\
\brokenvert\\
\lfloor
\end{array}\begin{array}{ccccccc}
a & m & -n\\
a+t & m-1 & -n-1 & -t & -t\\
a+t+v & m-2 & -n-2 & -t-1 & -t-1 & -v & -v
\end{array}\begin{array}{c}
\brokenvert\\
\brokenvert\\
\brokenvert
\end{array}$

$\begin{array}{cccccc}
-mn & 1\\
-mn+t & -t-1 & -t-1 & 1\\
-mn+t+v & -t-2 & -t-2 & -v-1 & -v-1 & 1
\end{array};\begin{array}{c}
1\\
1\\
1
\end{array}\begin{array}{c}
\rceil\\
\brokenvert\\
\rfloor
\end{array}.$

\end{enumerate}
\end{proof}
\begin{example}
The number of tableaux of battery shape $\left[\left(7^{11}\right),1,4\right]$

as in the figure 

$\left(\begin{array}{ccccccc}
 &  &  & \bullet\\
\bullet & \bullet & \bullet & \bullet & \bullet & \bullet & \bullet\\
\bullet & \bullet & \bullet & \bullet & \bullet & \bullet & \bullet\\
\bullet & \bullet & \bullet & \bullet & \bullet & \bullet & \bullet\\
\bullet & \bullet & \bullet & \bullet & \bullet & \bullet & \bullet\\
\bullet & \bullet & \bullet & \bullet & \bullet & \bullet & \bullet\\
\bullet & \bullet & \bullet & \bullet & \bullet & \bullet & \bullet\\
\bullet & \bullet & \bullet & \bullet & \bullet & \bullet & \bullet\\
\bullet & \bullet & \bullet & \bullet & \bullet & \bullet & \bullet\\
\bullet & \bullet & \bullet & \bullet & \bullet & \bullet & \bullet\\
\bullet & \bullet & \bullet & \bullet & \bullet & \bullet & \bullet\\
\bullet & \bullet & \bullet & \bullet & \bullet & \bullet & \bullet
\end{array}\right)$

is equal to:

$2^{7}\times3^{2}\times5^{2}\times7\times13\times17^{3}\times19^{3}\times23^{2}\times29^{2}\times31^{2}$

$\times37^{2}\times41\times43\times59\times61\times67\times71\times73\times2792843.$
\begin{example}
The number of $SYT$ of battery shape $\left[\left(m^{2}\right),a,4\right]$
is

\[
f^{\left(m^{2}\right)}\frac{\left(a+1\right)\left(a+2\right)\left(a+3\right)w}{1152\left(2m-1\right)\left(2m-3\right)\left(2m-5\right)}
\]

where

\[
w=a^{3}m\left(m-1\right)\left(m+1\right)+3a^{2}m\left(m+1\right)\left(11m-23\right)
\]

\[
+4a\left(m+1\right)\left(95m^{2}-338m+270\right)+192\left(2m-1\right)\left(2m-3\right)\left(2m-5\right).
\]
\end{example}

\end{example}

\subsection{Shape $\left[\left(m^{n}\right),a,5\right]$ }
\begin{thm}
The number of SYT of battery shape $\left[\left(m^{n}\right),a,5\right]$
is equal to 

$=f^{\left(m^{n}\right)}\mbox{}_{\left[3,5,7,9\right]}F_{\left[2,4,6,8\right]}$

$\begin{array}{c}
\lceil\\
\brokenvert\\
\brokenvert\\
\lfloor
\end{array}\begin{array}{ccccccccc}
a & m & -n\\
a+t & m-1 & -n-1 & -t & -t\\
a+t+v & m-2 & -n-2 & -t-1 & -t-1 & -v & -v\\
a+t+v+w & m-3 & -n-3 & -t-2 & -t-2 & -v-1 & -v-1 & -w & -w
\end{array}\begin{array}{c}
\brokenvert\\
\brokenvert\\
\brokenvert\\
\brokenvert
\end{array}$

$\begin{array}{cccccccc}
-mn & 1\\
-mn+t & -t-1 & -t-1 & 1\\
-mn+t+v & -t-2 & -t-2 & -v-1 & -v-1 & 1\\
-mn+t+v+w & -t-3 & -t-3 & -v-2 & -v-2 & -w-1 & -w-1 & 1
\end{array};\begin{array}{c}
1\\
1\\
1\\
1
\end{array}\begin{array}{c}
\rceil\\
\brokenvert\\
\brokenvert\\
\rfloor
\end{array}$
\end{thm}

\begin{proof}
We use the Hook Length Formula.
\begin{enumerate}
\item The number of bullet tableaux is

\[
\frac{\left(t+v+w+r\right)!\left(w-r+1\right)\left(v-w+1\right)\left(v-r+2\right)\left(t-v+1\right)\left(t-w+2\right)\left(t-r+3\right)}{r!\left(w+1\right)!\left(v+2\right)!\left(t+3\right)!}.
\]

\item The number of circle tableaux is

\[
f^{\left(m^{n}\right)}\frac{\left(\begin{array}{c}
t+m-1\\
t
\end{array}\right)\left(\begin{array}{c}
v+m-2\\
v
\end{array}\right)\left(\begin{array}{c}
w+m-3\\
w
\end{array}\right)\left(\begin{array}{c}
n+3\\
r
\end{array}\right)\left(\begin{array}{c}
t+2\\
r
\end{array}\right)\left(\begin{array}{c}
v+1\\
r
\end{array}\right)\left(\begin{array}{c}
w\\
r
\end{array}\right)}{\left(\begin{array}{c}
mn\\
t+v+w+r
\end{array}\right)\left(t+v+w+r\right)!\left(\begin{array}{c}
t+1\\
v
\end{array}\right)}
\]

\[
\frac{\left(\begin{array}{c}
r+m-4\\
r
\end{array}\right)\left(\begin{array}{c}
n\\
t
\end{array}\right)\left(\begin{array}{c}
n+1\\
v
\end{array}\right)\left(\begin{array}{c}
t\\
v
\end{array}\right)\left(\begin{array}{c}
v\\
w
\end{array}\right)}{\left(\begin{array}{c}
t+2\\
w
\end{array}\right)\left(\begin{array}{c}
v+1\\
w
\end{array}\right)\left(\begin{array}{c}
t+3\\
r
\end{array}\right)\left(\begin{array}{c}
v+2\\
r
\end{array}\right)\left(\begin{array}{c}
w+1\\
r
\end{array}\right)}
\]

\end{enumerate}
\[
\left(\begin{array}{c}
a+t+v+w+r-1\\
t+v+w+r
\end{array}\right)\left(\begin{array}{c}
n+2\\
w
\end{array}\right)\left(\begin{array}{c}
t+1\\
w
\end{array}\right)w!t!v!r!.
\]

The total number is therefore

\[
f^{\left(m^{n}\right)}\sum_{t=0}^{n}\sum_{v=0}^{t}\sum_{w=0}^{v}\sum_{r=0}^{w}\frac{\left(w-r+1\right)\left(v-w+1\right)\left(v-r+2\right)\left(t-v+1\right)\left(t-w+2\right)\left(t-r+3\right)}{12\left(\begin{array}{c}
mn\\
t+v+w+r
\end{array}\right)\left(\begin{array}{c}
t+1\\
v
\end{array}\right)\left(\begin{array}{c}
w+1\\
1
\end{array}\right)}
\]

$\frac{\left(\begin{array}{c}
t+m-1\\
t
\end{array}\right)\left(\begin{array}{c}
v+m-2\\
v
\end{array}\right)\left(\begin{array}{c}
w+m-3\\
w
\end{array}\right)\left(\begin{array}{c}
n+3\\
r
\end{array}\right)\left(\begin{array}{c}
t+2\\
r
\end{array}\right)\left(\begin{array}{c}
v+1\\
r
\end{array}\right)\left(\begin{array}{c}
w\\
r
\end{array}\right)}{\left(\begin{array}{c}
v+2\\
2
\end{array}\right)\left(\begin{array}{c}
t+3\\
3
\end{array}\right)\left(\begin{array}{c}
t+2\\
w
\end{array}\right)\left(\begin{array}{c}
v+1\\
w
\end{array}\right)}$

$\frac{\left(\begin{array}{c}
r+m-4\\
r
\end{array}\right)\left(\begin{array}{c}
n\\
t
\end{array}\right)\left(\begin{array}{c}
n+1\\
v
\end{array}\right)\left(\begin{array}{c}
t\\
v
\end{array}\right)\left(\begin{array}{c}
v\\
w
\end{array}\right)\left(\begin{array}{c}
a+t+v+w+r-1\\
t+v+w+r
\end{array}\right)\left(\begin{array}{c}
n+2\\
w
\end{array}\right)\left(\begin{array}{c}
t+1\\
w
\end{array}\right)}{\left(\begin{array}{c}
t+3\\
r
\end{array}\right)\left(\begin{array}{c}
v+2\\
r
\end{array}\right)\left(\begin{array}{c}
w+1\\
r
\end{array}\right)}$

$=f^{\left(m^{n}\right)}(\sum_{t=0}^{n}\sum_{v=0}^{t}\sum_{w=0}^{v}\sum_{r=0}^{w}\frac{\left(\begin{array}{c}
-m\\
t
\end{array}\right)\left(\begin{array}{c}
-m+1\\
v
\end{array}\right)\left(\begin{array}{c}
-m+2\\
w
\end{array}\right)}{\left(\begin{array}{c}
mn-t-v-w\\
r
\end{array}\right)\left(\begin{array}{c}
mn-t-v\\
w
\end{array}\right)\left(\begin{array}{c}
mn-t\\
v
\end{array}\right)}$

$\frac{\left(\begin{array}{c}
n+3\\
r
\end{array}\right)\left(\begin{array}{c}
t+2\\
r
\end{array}\right)^{2}\left(\begin{array}{c}
v+1\\
r
\end{array}\right)^{2}\left(\begin{array}{c}
w\\
r
\end{array}\right)^{2}}{\left(\begin{array}{c}
mn\\
t
\end{array}\right)\left(\begin{array}{c}
t+1\\
v
\end{array}\right)^{2}})$

$(\frac{\left(\begin{array}{c}
-m+3\\
r
\end{array}\right)\left(\begin{array}{c}
n\\
t
\end{array}\right)\left(\begin{array}{c}
n+1\\
v
\end{array}\right)\left(\begin{array}{c}
t\\
v
\end{array}\right)^{2}\left(\begin{array}{c}
v\\
w
\end{array}\right)^{2}\left(\begin{array}{c}
-a-t-v-w\\
r
\end{array}\right)}{\left(\begin{array}{c}
t+2\\
w
\end{array}\right)^{2}\left(\begin{array}{c}
v+1\\
w
\end{array}\right)^{2}\left(\begin{array}{c}
t+3\\
r
\end{array}\right)^{2}\left(\begin{array}{c}
v+2\\
r
\end{array}\right)^{2}\left(\begin{array}{c}
w+1\\
r
\end{array}\right)^{2}}$

$\left(\begin{array}{c}
-a-t-v\\
w
\end{array}\right)\left(\begin{array}{c}
-a-t\\
v
\end{array}\right)\left(\begin{array}{c}
-a\\
t
\end{array}\right)\left(\begin{array}{c}
n+2\\
w
\end{array}\right)\left(\begin{array}{c}
t+1\\
w
\end{array}\right)^{2})$

$=f^{\left(m^{n}\right)}\sum_{t=0}^{n}\sum_{v=0}^{t}\sum_{w=0}^{v}\frac{\left(\begin{array}{c}
-m\\
t
\end{array}\right)\left(\begin{array}{c}
-m+1\\
v
\end{array}\right)\left(\begin{array}{c}
-m+2\\
w
\end{array}\right)}{\left(\begin{array}{c}
mn-t-v\\
w
\end{array}\right)\left(\begin{array}{c}
mn-t\\
v
\end{array}\right)\left(\begin{array}{c}
mn\\
t
\end{array}\right)\left(\begin{array}{c}
t+1\\
v
\end{array}\right)^{2}}$

$\frac{\left(\begin{array}{c}
n\\
t
\end{array}\right)\left(\begin{array}{c}
n+1\\
v
\end{array}\right)\left(\begin{array}{c}
t\\
v
\end{array}\right)^{2}\left(\begin{array}{c}
v\\
w
\end{array}\right)^{2}\left(\begin{array}{c}
-a-t-v\\
w
\end{array}\right)\left(\begin{array}{c}
-a-t\\
v
\end{array}\right)\left(\begin{array}{c}
-a\\
t
\end{array}\right)\left(\begin{array}{c}
n+2\\
w
\end{array}\right)\left(\begin{array}{c}
t+1\\
w
\end{array}\right)^{2}}{\left(\begin{array}{c}
t+2\\
w
\end{array}\right)^{2}\left(\begin{array}{c}
v+1\\
w
\end{array}\right)^{2}}$

$\sum_{r=0}^{w}\frac{\left(\begin{array}{c}
-m+3\\
r
\end{array}\right)\left(\begin{array}{c}
-a-t-v-w\\
r
\end{array}\right)\left(\begin{array}{c}
n+3\\
r
\end{array}\right)\left(\begin{array}{c}
t+2\\
r
\end{array}\right)^{2}\left(\begin{array}{c}
v+1\\
r
\end{array}\right)^{2}\left(\begin{array}{c}
w\\
r
\end{array}\right)^{2}r!}{\left(\begin{array}{c}
t+3\\
r
\end{array}\right)^{2}\left(\begin{array}{c}
v+2\\
r
\end{array}\right)^{2}\left(\begin{array}{c}
w+1\\
r
\end{array}\right)^{2}\left(\begin{array}{c}
mn-t-v-w\\
r
\end{array}\right)\left(\begin{array}{c}
r\\
r
\end{array}\right)}\frac{1}{r!}$

$=f^{\left(m^{n}\right)}\mbox{}_{\left[3,5,7,9\right]}F_{\left[2,4,6,8\right]}$

$\begin{array}{c}
\lceil\\
\brokenvert\\
\brokenvert\\
\lfloor
\end{array}\begin{array}{ccccccccc}
a & m & -n\\
a+t & m-1 & -n-1 & -t & -t\\
a+t+v & m-2 & -n-2 & -t-1 & -t-1 & -v & -v\\
a+t+v+w & m-3 & -n-3 & -t-2 & -t-2 & -v-1 & -v-1 & -w & -w
\end{array}\begin{array}{c}
\brokenvert\\
\brokenvert\\
\brokenvert\\
\brokenvert
\end{array}$

$\begin{array}{cccccccc}
-mn & 1\\
-mn+t & -t-1 & -t-1 & 1\\
-mn+t+v & -t-2 & -t-2 & -v-1 & -v-1 & 1\\
-mn+t+v+w & -t-3 & -t-3 & -v-2 & -v-2 & -w-1 & -w-1 & 1
\end{array};\begin{array}{c}
1\\
1\\
1\\
1
\end{array}\begin{array}{c}
\rceil\\
\brokenvert\\
\brokenvert\\
\rfloor
\end{array}$

\end{proof}
\begin{example}
The number of $SYT$ of battery shape $\left[\left(m^{2}\right),a,5\right]$
is

\[
f^{\left(m^{2}\right)}\frac{\left(a+1\right)\left(a+2\right)\left(a+3\right)\left(a+4\right)w}{46080\left(2m-1\right)\left(2m-3\right)\left(2m-5\right)\left(2m-7\right)}
\]

where

\[
w=a^{4}m\left(m-1\right)\left(m+1\right)\left(m-2\right)+2a^{3}m\left(m+1\right)\left(m-1\right)\left(25m-74\right)
\]

\[
+a^{2}\left(m+1\right)m\left(m-2\right)\left(971m-3131\right)+2a\left(m+1\right)\left(4361m^{3}-29979m^{2}+63418m-40320\right)
\]

\[
+1920\left(2m-1\right)\left(2m-3\right)\left(2m-5\right)\left(2m-7\right)
\]
\end{example}

\subsection{Shape $\left[\left(m^{n}\right),a,6\right]$ }
\begin{thm}
The number of SYT of battery shape $\left[\left(m^{n}\right),a,6\right]$
is equal to 

$f^{\left(m^{n}\right)}\mbox{}_{\left[3,5,7,9,11\right]}F_{\left[2,4,6,8,10\right]}$

$\begin{array}{c}
\lceil\\
\brokenvert\\
\brokenvert\\
\lfloor
\end{array}\begin{array}{cccccc}
a & m & -n\\
a+t & m-1 & -n-1 & -t & -t\\
a+t+v & m-2 & -n-2 & -t-1 & t-1 & -v\\
a+t+v+w & m-3 & -n-3 & -t-2 & -t-2 & -v-1\\
a+t+v+w+r & m-4 & -n-4 & -t-3 & -t-3 & -v-2
\end{array}$

$\begin{array}{cccccc}
 &  &  &  &  & \brokenvert\\
 &  &  &  &  & \brokenvert\\
-v &  &  &  &  & \brokenvert\\
-v-1 & -w & -w &  &  & \brokenvert\\
-v-2 & -w-1 & -w-1 & -r & -r & \brokenvert
\end{array}$

$\begin{array}{c}
\lceil\\
\brokenvert\\
\brokenvert\\
\brokenvert\\
\lfloor
\end{array}\begin{array}{ccccc}
-mn & 1\\
-mn+t & -t-1 & -t-1 & 1\\
-mn+t+v & -t-2 & -t-2 & -v-1 & -v-1\\
-mn+t+v+w & -t-3 & -t-3 & -v-2 & -v-2\\
-mn+t+v+w+r & -t-4 & -t-4 & -v-3 & -v-3
\end{array}$

$\begin{array}{ccccc}
\\
\\
1\\
-w-1 & -w-1 & 1\\
-w-2 & -w-2 & -r-1 & -r-1 & 1
\end{array};\begin{array}{c}
1\\
1\\
1\\
1\\
1
\end{array}\begin{array}{c}
\rceil\\
\brokenvert\\
\brokenvert\\
\brokenvert\\
\rfloor
\end{array}.$
\end{thm}

\begin{proof}
We use the Hook Length Formula.
\begin{enumerate}
\item The number of bullet tableaux is

\[
\frac{\left(t+v+w+r+k\right)!\left(r-k+1\right)\left(w-r+1\right)\left(w-k+2\right)\left(v-w+1\right)\left(v-r+2\right)\left(v-k+3\right)}{k!\left(r+1\right)!\left(w+2\right)!\left(v+3\right)!\left(t+4\right)!}.
\]

$\left(\left(t-v+1\right)\left(t-w+2\right)\left(t-r+3\right)\left(t-k+4\right)\right)$
\begin{enumerate}
\item The number of circle tableaux is
\end{enumerate}
\[
f^{\left(m^{n}\right)}\frac{\left(\begin{array}{c}
n+1\\
v
\end{array}\right)\left(\begin{array}{c}
n+2\\
w
\end{array}\right)\left(\begin{array}{c}
t+1\\
w
\end{array}\right)\left(\begin{array}{c}
n+3\\
r
\end{array}\right)\left(\begin{array}{c}
t+2\\
r
\end{array}\right)\left(\begin{array}{c}
v+1\\
r
\end{array}\right)}{\left(\begin{array}{c}
mn\\
t+v+w+r+k
\end{array}\right)\left(t+v+w+r+k\right)!\left(\begin{array}{c}
t+1\\
v
\end{array}\right)\left(\begin{array}{c}
t+2\\
w
\end{array}\right)\left(\begin{array}{c}
v+1\\
w
\end{array}\right)}
\]

\[
\frac{\left(\begin{array}{c}
n+4\\
k
\end{array}\right)\left(\begin{array}{c}
t+3\\
k
\end{array}\right)\left(\begin{array}{c}
v+2\\
k
\end{array}\right)\left(\begin{array}{c}
w+1\\
k
\end{array}\right)\left(\begin{array}{c}
n\\
t
\end{array}\right)\left(\begin{array}{c}
t\\
v
\end{array}\right)\left(\begin{array}{c}
v\\
w
\end{array}\right)\left(\begin{array}{c}
w\\
r
\end{array}\right)\left(\begin{array}{c}
r\\
k
\end{array}\right)k!r!w!v!t!}{\left(\begin{array}{c}
w+1\\
r
\end{array}\right)\left(\begin{array}{c}
v+2\\
r
\end{array}\right)\left(\begin{array}{c}
t+3\\
r
\end{array}\right)\left(\begin{array}{c}
t+4\\
k
\end{array}\right)\left(\begin{array}{c}
v+3\\
k
\end{array}\right)\left(\begin{array}{c}
w+2\\
k
\end{array}\right)\left(\begin{array}{c}
r+1\\
k
\end{array}\right)}
\]

\[
\left(\begin{array}{c}
t+m-1\\
t
\end{array}\right)\left(\begin{array}{c}
v+m-2\\
v
\end{array}\right)\left(\begin{array}{c}
w+m-3\\
w
\end{array}\right)\left(\begin{array}{c}
r+m-4\\
r
\end{array}\right)\left(\begin{array}{c}
k+m-5\\
k
\end{array}\right)
\]
.

The total number is therefore

$f^{\left(m^{n}\right)}(\sum_{t=0}^{n}\sum_{v=0}^{t}\sum_{w=0}^{v}\sum_{r=0}^{w}\sum_{k=0}^{r}$
\end{enumerate}
\[
\frac{\left(r-k+1\right)\left(w-r+1\right)\left(w-k+2\right)\left(v-w+1\right)\left(v-r+2\right)\left(v-k+3\right)}{\left(\begin{array}{c}
mn\\
t+v+w+r+k
\end{array}\right)\left(\begin{array}{c}
t+1\\
v
\end{array}\right)\left(\begin{array}{c}
t+2\\
w
\end{array}\right)\left(\begin{array}{c}
v+1\\
w
\end{array}\right)}
\]
\[
\frac{\left(t-v+1\right)\left(t-w+2\right)\left(t-r+3\right)\left(t-k+4\right)}{\left(\begin{array}{c}
w+1\\
r
\end{array}\right)\left(\begin{array}{c}
v+2\\
r
\end{array}\right)\left(\begin{array}{c}
t+3\\
r
\end{array}\right)}
\]

$\frac{\left(\begin{array}{c}
t+m-1\\
t
\end{array}\right)\left(\begin{array}{c}
v+m-2\\
v
\end{array}\right)\left(\begin{array}{c}
w+m-3\\
w
\end{array}\right)\left(\begin{array}{c}
r+m-4\\
r
\end{array}\right)\left(\begin{array}{c}
k+m-5\\
k
\end{array}\right)}{\left(\begin{array}{c}
t+4\\
k
\end{array}\right)\left(\begin{array}{c}
v+3\\
k
\end{array}\right)\left(\begin{array}{c}
w+2\\
k
\end{array}\right)\left(\begin{array}{c}
r+1\\
k
\end{array}\right)}$

$\frac{\left(\begin{array}{c}
a+t+v+w+r+k-1\\
t+v+w+r+k
\end{array}\right)\left(\begin{array}{c}
n\\
t
\end{array}\right)\left(\begin{array}{c}
t\\
v
\end{array}\right)\left(\begin{array}{c}
v\\
w
\end{array}\right)\left(\begin{array}{c}
w\\
r
\end{array}\right)\left(\begin{array}{c}
r\\
k
\end{array}\right)}{288\left(\begin{array}{c}
r+1\\
1
\end{array}\right)\left(\begin{array}{c}
w+2\\
2
\end{array}\right)\left(\begin{array}{c}
v+3\\
3
\end{array}\right)\left(\begin{array}{c}
t+4\\
4
\end{array}\right)}$

$\left(\left(\begin{array}{c}
n+1\\
v
\end{array}\right)\left(\begin{array}{c}
n+2\\
w
\end{array}\right)\left(\begin{array}{c}
t+1\\
w
\end{array}\right)\left(\begin{array}{c}
n+3\\
r
\end{array}\right)\left(\begin{array}{c}
t+2\\
r
\end{array}\right)\left(\begin{array}{c}
v+1\\
r
\end{array}\right)\right)$

$\left(\left(\begin{array}{c}
n+4\\
k
\end{array}\right)\left(\begin{array}{c}
t+3\\
k
\end{array}\right)\left(\begin{array}{c}
v+2\\
k
\end{array}\right)\left(\begin{array}{c}
w+1\\
k
\end{array}\right)\right)$

=$f^{\left(m^{n}\right)}(\sum_{t=0}^{n}\sum_{v=0}^{t}\sum_{w=0}^{v}\sum_{r=0}^{w}\sum_{k=0}^{r}$

$\frac{\left(\begin{array}{c}
t+m-1\\
t
\end{array}\right)\left(\begin{array}{c}
v+m-2\\
v
\end{array}\right)\left(\begin{array}{c}
w+m-3\\
w
\end{array}\right)\left(\begin{array}{c}
r+m-4\\
r
\end{array}\right)\left(\begin{array}{c}
k+m-5\\
k
\end{array}\right)}{\left(\begin{array}{c}
mn\\
t+v+w+r+k
\end{array}\right)\left(\begin{array}{c}
t+1\\
v
\end{array}\right)^{2}\left(\begin{array}{c}
t+2\\
w
\end{array}\right)^{2}\left(\begin{array}{c}
v+1\\
w
\end{array}\right)^{2}}$

$\frac{\left(\begin{array}{c}
a+t+v+w+r+k-1\\
t+v+w+r+k
\end{array}\right)\left(\begin{array}{c}
n\\
t
\end{array}\right)\left(\begin{array}{c}
t\\
v
\end{array}\right)^{2}\left(\begin{array}{c}
v\\
w
\end{array}\right)^{2}\left(\begin{array}{c}
w\\
r
\end{array}\right)^{2}\left(\begin{array}{c}
r\\
k
\end{array}\right)^{2}}{\left(\begin{array}{c}
w+1\\
r
\end{array}\right)^{2}\left(\begin{array}{c}
v+2\\
r
\end{array}\right)^{2}\left(\begin{array}{c}
t+3\\
r
\end{array}\right)^{2}}$

$\frac{\left(\begin{array}{c}
n+1\\
v
\end{array}\right)\left(\begin{array}{c}
n+2\\
w
\end{array}\right)\left(\begin{array}{c}
t+1\\
w
\end{array}\right)^{2}\left(\begin{array}{c}
n+3\\
r
\end{array}\right)\left(\begin{array}{c}
t+2\\
r
\end{array}\right)^{2}\left(\begin{array}{c}
v+1\\
r
\end{array}\right)^{2}}{\left(\begin{array}{c}
t+4\\
k
\end{array}\right)^{2}\left(\begin{array}{c}
v+3\\
k
\end{array}\right)^{2}\left(\begin{array}{c}
w+2\\
k
\end{array}\right)^{2}\left(\begin{array}{c}
r+1\\
k
\end{array}\right)^{2}}$

$\left(\begin{array}{c}
n+4\\
k
\end{array}\right)\left(\begin{array}{c}
t+3\\
k
\end{array}\right)^{2}\left(\begin{array}{c}
v+2\\
k
\end{array}\right)^{2}\left(\begin{array}{c}
w+1\\
k
\end{array}\right)^{2}$

=$f^{\left(m^{n}\right)}(\sum_{t=0}^{n}\sum_{v=0}^{t}\sum_{w=0}^{v}\sum_{r=0}^{w}\sum_{k=0}^{r}$

$\frac{\left(\begin{array}{c}
-m\\
t
\end{array}\right)\left(\begin{array}{c}
-m+1\\
v
\end{array}\right)\left(\begin{array}{c}
-m+2\\
w
\end{array}\right)\left(\begin{array}{c}
-m+3\\
r
\end{array}\right)\left(\begin{array}{c}
-m+4\\
k
\end{array}\right)}{\left(\begin{array}{c}
mn-t-v-w-r\\
k
\end{array}\right)\left(\begin{array}{c}
mn-t-v-w\\
r
\end{array}\right)\left(\begin{array}{c}
mn-t-v\\
w
\end{array}\right)\left(\begin{array}{c}
mn-t\\
v
\end{array}\right)}$

$\frac{\left(\begin{array}{c}
-a-t-v-w-r\\
k
\end{array}\right)\left(\begin{array}{c}
-a-t-v-w\\
r
\end{array}\right)\left(\begin{array}{c}
-a-t-v\\
w
\end{array}\right)}{\left(\begin{array}{c}
mn\\
t
\end{array}\right)\left(\begin{array}{c}
t+1\\
v
\end{array}\right)^{2}\left(\begin{array}{c}
t+2\\
w
\end{array}\right)^{2}\left(\begin{array}{c}
v+1\\
w
\end{array}\right)^{2}}$

$\frac{\left(\begin{array}{c}
-a-t\\
v
\end{array}\right)\left(\begin{array}{c}
-a\\
t
\end{array}\right)\left(\begin{array}{c}
n\\
t
\end{array}\right)\left(\begin{array}{c}
t\\
v
\end{array}\right)^{2}\left(\begin{array}{c}
v\\
w
\end{array}\right)^{2}\left(\begin{array}{c}
w\\
r
\end{array}\right)^{2}\left(\begin{array}{c}
r\\
k
\end{array}\right)^{2}}{\left(\begin{array}{c}
w+1\\
r
\end{array}\right)^{2}\left(\begin{array}{c}
v+2\\
r
\end{array}\right)^{2}\left(\begin{array}{c}
t+3\\
r
\end{array}\right)^{2}}$

$\frac{\left(\begin{array}{c}
n+1\\
v
\end{array}\right)\left(\begin{array}{c}
n+2\\
w
\end{array}\right)\left(\begin{array}{c}
t+1\\
w
\end{array}\right)^{2}\left(\begin{array}{c}
n+3\\
r
\end{array}\right)\left(\begin{array}{c}
t+2\\
r
\end{array}\right)^{2}\left(\begin{array}{c}
v+1\\
r
\end{array}\right)^{2}}{\left(\begin{array}{c}
t+4\\
k
\end{array}\right)^{2}\left(\begin{array}{c}
v+3\\
k
\end{array}\right)^{2}\left(\begin{array}{c}
w+2\\
k
\end{array}\right)^{2}\left(\begin{array}{c}
r+1\\
k
\end{array}\right)^{2}}$

$\left(\begin{array}{c}
n+4\\
k
\end{array}\right)\left(\begin{array}{c}
t+3\\
k
\end{array}\right)^{2}\left(\begin{array}{c}
v+2\\
k
\end{array}\right)^{2}\left(\begin{array}{c}
w+1\\
k
\end{array}\right)^{2}$

\break

=$f^{\left(m^{n}\right)}(\sum_{t=0}^{n}\sum_{v=0}^{t}\sum_{w=0}^{v}\sum_{r=0}^{w}$

$\frac{\left(\begin{array}{c}
n+1\\
v
\end{array}\right)\left(\begin{array}{c}
n+2\\
w
\end{array}\right)\left(\begin{array}{c}
t+1\\
w
\end{array}\right)^{2}\left(\begin{array}{c}
n+3\\
r
\end{array}\right)\left(\begin{array}{c}
t+2\\
r
\end{array}\right)^{2}\left(\begin{array}{c}
v+1\\
r
\end{array}\right)^{2}}{\left(\begin{array}{c}
mn-t-v-w\\
r
\end{array}\right)\left(\begin{array}{c}
mn-t-v\\
w
\end{array}\right)\left(\begin{array}{c}
mn-t\\
v
\end{array}\right)\left(\begin{array}{c}
mn\\
t
\end{array}\right)\left(\begin{array}{c}
t+1\\
v
\end{array}\right)^{2}\left(\begin{array}{c}
t+2\\
w
\end{array}\right)^{2}\left(\begin{array}{c}
v+1\\
w
\end{array}\right)^{2}}$

$\frac{\left(\begin{array}{c}
-a-t-v-w\\
r
\end{array}\right)\left(\begin{array}{c}
-a-t-v\\
w
\end{array}\right)\left(\begin{array}{c}
-a-t\\
v
\end{array}\right)\left(\begin{array}{c}
-a\\
t
\end{array}\right)\left(\begin{array}{c}
n\\
t
\end{array}\right)\left(\begin{array}{c}
t\\
v
\end{array}\right)^{2}\left(\begin{array}{c}
v\\
w
\end{array}\right)^{2}\left(\begin{array}{c}
w\\
r
\end{array}\right)^{2}}{\left(\begin{array}{c}
w+1\\
r
\end{array}\right)^{2}\left(\begin{array}{c}
v+2\\
r
\end{array}\right)^{2}\left(\begin{array}{c}
t+3\\
r
\end{array}\right)^{2}}$

$\left(\begin{array}{c}
-m\\
t
\end{array}\right)\left(\begin{array}{c}
-m+1\\
v
\end{array}\right)\left(\begin{array}{c}
-m+2\\
w
\end{array}\right)\left(\begin{array}{c}
-m+3\\
r
\end{array}\right)$

$\sum_{k=0}^{r}\frac{\left(\begin{array}{c}
-m+4\\
k
\end{array}\right)\left(\begin{array}{c}
-a-t-v-w-r\\
k
\end{array}\right)\left(\begin{array}{c}
r\\
k
\end{array}\right)^{2}\left(\begin{array}{c}
n+4\\
k
\end{array}\right)\left(\begin{array}{c}
t+3\\
k
\end{array}\right)^{2}\left(\begin{array}{c}
v+2\\
k
\end{array}\right)^{2}\left(\begin{array}{c}
w+1\\
k
\end{array}\right)^{2}k!}{\left(\begin{array}{c}
mn-t-v-w-r\\
k
\end{array}\right)\left(\begin{array}{c}
t+4\\
k
\end{array}\right)^{2}\left(\begin{array}{c}
v+3\\
k
\end{array}\right)^{2}\left(\begin{array}{c}
w+2\\
k
\end{array}\right)^{2}\left(\begin{array}{c}
r+1\\
k
\end{array}\right)^{2}\left(\begin{array}{c}
k\\
k
\end{array}\right)}\frac{1}{k!}$

$=f^{\left(m^{n}\right)}\mbox{}_{\left[3,5,7,9,11\right]}F_{\left[2,4,6,8,10\right]}$

$\begin{array}{c}
\lceil\\
\brokenvert\\
\brokenvert\\
\lfloor
\end{array}\begin{array}{cccccc}
a & m & -n\\
a+t & m-1 & -n-1 & -t & -t\\
a+t+v & m-2 & -n-2 & -t-1 & t-1 & -v\\
a+t+v+w & m-3 & -n-3 & -t-2 & -t-2 & -v-1\\
a+t+v+w+r & m-4 & -n-4 & -t-3 & -t-3 & -v-2
\end{array}$

$\begin{array}{cccccc}
 &  &  &  &  & \brokenvert\\
 &  &  &  &  & \brokenvert\\
-v &  &  &  &  & \brokenvert\\
-v-1 & -w & -w &  &  & \brokenvert\\
-v-2 & -w-1 & -w-1 & -r & -r & \brokenvert
\end{array}$

$\begin{array}{c}
\lceil\\
\brokenvert\\
\brokenvert\\
\brokenvert\\
\lfloor
\end{array}\begin{array}{ccccc}
-mn & 1\\
-mn+t & -t-1 & -t-1 & 1\\
-mn+t+v & -t-2 & -t-2 & -v-1 & -v-1\\
-mn+t+v+w & -t-3 & -t-3 & -v-2 & -v-2\\
-mn+t+v+w+r & -t-4 & -t-4 & -v-3 & -v-3
\end{array}$

$\begin{array}{ccccc}
\\
\\
1\\
-w-1 & -w-1 & 1\\
-w-2 & -w-2 & -r-1 & -r-1 & 1
\end{array};\begin{array}{c}
1\\
1\\
1\\
1\\
1
\end{array}\begin{array}{c}
\rceil\\
\brokenvert\\
\brokenvert\\
\brokenvert\\
\rfloor
\end{array}.$
\end{proof}

The following example answers an enumerative chess problem asked by
Buchanan \cite{Bu} .

\break
\begin{example}
The number of SYT of battery shape $\left[\left(11^{7}\right),1,6\right]$

as in the figure 

$\left(\begin{array}{ccccccccccc}
 &  &  &  &  & \bullet\\
\bullet & \bullet & \bullet & \bullet & \bullet & \bullet & \bullet & \bullet & \bullet & \bullet & \bullet\\
\bullet & \bullet & \bullet & \bullet & \bullet & \bullet & \bullet & \bullet & \bullet & \bullet & \bullet\\
\bullet & \bullet & \bullet & \bullet & \bullet & \bullet & \bullet & \bullet & \bullet & \bullet & \bullet\\
\bullet & \bullet & \bullet & \bullet & \bullet & \bullet & \bullet & \bullet & \bullet & \bullet & \bullet\\
\bullet & \bullet & \bullet & \bullet & \bullet & \bullet & \bullet & \bullet & \bullet & \bullet & \bullet\\
\bullet & \bullet & \bullet & \bullet & \bullet & \bullet & \bullet & \bullet & \bullet & \bullet & \bullet\\
\bullet & \bullet & \bullet & \bullet & \bullet & \bullet & \bullet & \bullet & \bullet & \bullet & \bullet
\end{array}\right).$

is equal to:

$2^{5}\times3^{2}\times5^{2}\times11\times13\times17^{2}\times19^{3}\times23^{2}\times29\times31$$\times37^{2}\times41\times3361178017$

$\times2839893182041.$
\end{example}

\end{document}